\documentclass[pdflatex,sn-mathphys-num]{sn-jnl}
\usepackage{graphicx}
\usepackage{amsmath,amssymb,amsfonts}
\usepackage{amsthm}
\usepackage{manyfoot}
\usepackage{bm}
\usepackage{algorithm}
\usepackage{algorithmic}

\theoremstyle{definition}
\newtheorem{theorem}{Theorem}[section]
\newtheorem{lemma}[theorem]{Lemma}

\newtheorem{corollary}[theorem]{Corollary}

\newcommand{\diag}{\ensuremath\mathop{\mathrm{diag}}}
\newcommand{\fl}{\ensuremath\mathop{\mathrm{fl}}}
\newcommand{\FL}{\ensuremath\mathop{\mathrm{FL}}}
\newcommand{\fma}{\ensuremath\mathop{\mathrm{fma}}}
\newcommand{\hypot}{\ensuremath\mathop{\mathrm{hypot}}}
\newcommand{\sign}{\ensuremath\mathop{\mathrm{sign}}}
\begin{document}
\title[Enhancing the $2\times 2$ Kogbetliantz SVD]{Arithmetical enhancements of the Kogbetliantz method for the SVD of order two}
\author*{\fnm{Vedran}~\sur{Novakovi\'{c}}}\email{venovako@venovako.eu}
\equalcont{\url{https://orcid.org/0000-0003-2964-9674}}
\affil*{\orgdiv{independent} \orgname{researcher}, \orgaddress{\street{Vankina ulica 15}, \state{HR}-\postcode{10020} \city{Zagreb}, \country{Croatia}}}
\abstract{An enhanced Kogbetliantz method for the singular value
  decomposition (SVD) of general matrices of order two is proposed.
  The method consists of three phases: an almost exact prescaling,
  that can be beneficial to the LAPACK's \texttt{xLASV2} routine for
  the SVD of upper triangular $2\times 2$ matrices as well, a highly
  relatively accurate triangularization in the absence of underflows,
  and an alternative procedure for computing the SVD of triangular
  matrices, that employs the correctly rounded $\hypot$ function.  A
  heuristic for improving numerical orthogonality of the left singular
  vectors is also presented and tested on a wide spectrum of random
  input matrices.  On upper triangular matrices under test, the
  proposed method, unlike \texttt{xLASV2}, finds both singular values
  with high relative accuracy as long as the input elements are within
  a safe range that is almost as wide as the entire normal range.  On
  general matrices of order two, the method's safe range for which the
  smaller singular values remain accurate is of about half the width
  of the normal range.}
\maketitle
\section{Introduction}\label{s:1}
The singular value decomposition (SVD) of a square matrix of order two
is a widely used numerical tool.  In LAPACK~\cite{Anderson-et-al-99}
alone, its \texttt{xLASV2} routine for the SVD of real upper
triangular $2\times 2$ matrices is a building block for the QZ
algorithm~\cite{Moler-Stewart-73} for the generalized eigenvalue
problem $Ax=\lambda Bx$, with $A$ and $B$ real and square, and for the
SVD of real bidiagonal matrices by the implicit QR
algorithm~\cite{Demmel-Kahan-90}.  Also, the oldest method for the SVD
of square matrices that is still in use was developed by
Kogbetliantz~\cite{Kogbetliantz-55}, based on the SVD of order two,
and as such is the primary motivation for this research.

This work explores how to compute the SVD of a \emph{general} matrix
of order two indirectly, by a careful scaling, a highly relatively
accurate triangularization if the matrix indeed contains no zeros, and
an alternative triangular SVD method, since the straightforward
formulas for general matrices are challenging to be evaluated stably.

Let $G$ be a square real matrix of order $n$.  The SVD of $G$ is a
decomposition $G=U\Sigma V^T$, where $U$ and $V$ are
orthogonal\footnote{If $G$ is complex, $U$ and $V$ are unitary and
$G=U\Sigma V^{\ast}$, but this case is only briefly dealt with here.}
matrices of order $n$ of the left and the right singular vectors of
$G$, respectively, and $\Sigma=\diag(\sigma_1^{},\ldots,\sigma_n)$ is
a diagonal matrix of its singular values, such that
$\sigma_i^{}\ge\sigma_j^{}\ge 0$ for all $i$ and $j$ where
$1\le i<j\le n$.

In the step $k$ of the Kogbetliantz SVD method, a pivot submatrix of
order two (or several of them, not sharing indices with each other,
if the method is parallel) is found according to the chosen pivot
strategy in the iteration matrix $G_k$, its SVD is computed, and
$U_k$, $V_k$, and $G_k$ are updated by the transformation matrices
$\mathsf{U}_k$ and/or $\mathsf{V}_k$, leaving zeros in the
off-diagonal positions $(j_k,i_k)$ and $(i_k,j_k)$ of $G_{k+1}$, as in
\begin{equation}
  \begin{gathered}
    G_0=\mathop{\mathrm{preprocess}}(G),\quad
    U_0=\mathop{\mathrm{preprocess}}(I),\quad
    V_0=\mathop{\mathrm{preprocess}}(I);\\
    G_{k+1}^{}=\mathsf{U}_k^T G_k^{} \mathsf{V}_k^{},\quad
    U_{k+1}^{}=U_k^{}\mathsf{U}_k^{},\quad
    V_{k+1}^{}=V_k^{}\mathsf{V}_k^{},\qquad
    k\ge 0;\\
    \mathop{\mathrm{convergence}}(k=K)\implies
    U\approx U_K^{},\quad
    V\approx V_K^{},\quad
    \sigma_i^{}\approx g_{ii}^{(K)},\quad
    1\le i\le n.
  \end{gathered}
  \label{e:SVD}
\end{equation}
The left and the right singular vectors of the pivot matrix are
embedded into identities to get $\mathsf{U}_k$ and $\mathsf{V}_k$,
respectively, with the index mapping from matrices of order two to
$\mathsf{U}_k$ and $\mathsf{V}_k$ being $(1,1)\mapsto(i_k,i_k)$,
$(2,1)\mapsto(j_k,i_k)$, $(1,2)\mapsto(i_k,j_k)$,
$(2,2)\mapsto(j_k,j_k)$, where $1\le i_k<j_k\le n$ are the pivot
indices.  The process is repeated until convergence, i.e., until for
some $k=K$ the off-diagonal norm of $G_K$ falls below a certain
threshold.

If $G$ has $m$ rows and $n$ columns, $m>n$, it should be
preprocessed~\cite{Charlier-et-al-87} to a square matrix $G_0^{}$ by a
factorization of the URV~\cite{Stewart-92} type (e.g., the QR
factorization with column pivoting~\cite{Quintana-Orti-et-al-98}).
Then, $U_0^TGV_0^{}=G_0^{}$, where $G_0^{}$ is triangular of order $n$
and $U_0$ is orthogonal of order $m$.  If $m<n$, then the SVD of $G^T$
can be computed instead.

In all iterations it would be beneficial to have the pivot matrix
$\widehat{G}_k$ triangular, since its SVD can be computed with high
relative accuracy under mild assumptions~\cite{Hari-Matejas-09}.  This
is however not possible with time consuming but simple,
quadratically convergent~\cite[Remark~6]{Charlier-VanDooren-87} pivot
strategy that chooses the pivot with the largest off-diagonal norm
$\sqrt{|g_{j_ki_k}^{(k)}|^2+|g_{i_kj_k}^{(k)}|^2}$, but is possible,
if $G_0$ is triangular, with certain sequential cyclic (i.e.,
periodic)
strategies~\cite{Charlier-et-al-87,Hari-Veselic-87} like the
row-cyclic and column-cyclic, and even with some parallel ones, after
further preprocessing $G_0$ into a suitable ``butterfly''
form~\cite{Hari-Zadelj-Martic-07}.

Although the row-cyclic and column-cyclic strategies ensure
global~\cite{Hari-Veselic-87} and asymptotically
quadratic~\cite{Charlier-VanDooren-87,Hari-91} convergence of the
method, as well as its high relative accuracy~\cite{Matejas-Hari-15},
the method's sequential variants remain slow on modern hardware, while
preprocessing $G$ to $G_0$ (in the butterfly form or not) can only be
partially parallelized.

This work is a part of a broader
effort~\cite{Novakovic-20,Novakovic-Singer-22} to investigate if a
fast and accurate (in practice if not in theory) variant of the
Kogbetliantz method could be developed, that would be entirely
parallel and would function on general square matrices without
expensive preprocessing, with full pivots $\widehat{G}_k^{[\ell]}$,
$1\le\ell\le\mathsf{n}\le\lfloor n/2\rfloor$, that are independently
diagonalized, and with $\mathsf{n}$ ensuing concurrent updates of
$U_k$, $\mathsf{n}$ of $V_k$, and $\mathsf{n}$ from each of the sides
of $G_k$ in a parallel step.  This way the employed parallel pivot
strategy does not have to be cyclic.  A promising candidate is the
dynamic
ordering~\cite{Becka-et-al-02,Oksa-et-al-22,Novakovic-Singer-22}.

The proposed Kogbetliantz SVD of order two supports a wider exponent
range of the elements of a triangular input matrix for which both
singular values are computed with high relative accuracy than
\texttt{xLAEV2}, although the latter is slightly more accurate when
comparison is possible.  Matrices of the singular vectors obtained by
the proposed method are noticeably more numerically orthogonal.  With
respect to~\cite{Novakovic-20,Novakovic-Singer-22} and a general
matrix $G$ of order two, the following enhancements have been
implemented:
\begin{enumerate}
\item The structure of $G$ is exploited to the utmost extent, so the
  triangularization and a stronger scaling is employed only when $G$
  has no zeros, thus preserving accuracy.
\item The triangularization of $G$ by a special URV factorization is
  tweaked so that high relative accuracy of each computed element is
  provable when no underflow occurs.
\item The SVD procedure for triangular matrices utilizes the latest
  advances in computing the correctly rounded functions, so the
  pertinent formulas from~\cite{Novakovic-20,Novakovic-Singer-22} are
  altered.
\item The left singular vectors are computed by a heuristic when the
  triangularization is involved, by composing the two plane
  rotations---one from the URV factorization, and the other from the
  triangular SVD---into one without the matrix multiplication.
\end{enumerate}
High relative accuracy of the singular values of $G$ is observed, but
not proved, when the range of the elements of $G$ is narrower than
about half of the entire normal range.

\looseness=-1
This paper is organized as follows.  In Section~\ref{s:2} the
floating-point environment and the required operations are described,
and some auxiliary results regarding them are proved.
Section~\ref{s:3} presents the proposed SVD method.  In
Section~\ref{s:4} the numerical testing results are shown.
Section~\ref{s:5} concludes the paper with the comments on future
work.
\section{Floating-point considerations}\label{s:2}
Let $x$ be a real, infinite, or undefined (Not-a-Number) value:
$x\in\mathbb{R}\cup\{-\infty,+\infty,\mathrm{NaN}\}$.  Its floating-point
representation is denoted by $\fl(x)$ and is obtained by rounding $x$
to a value of the chosen standard floating-point datatype using
the rounding mode in effect, that is here assumed to be to
\emph{nearest} (ties to even).  If the result is normal,
$\fl(x)=x(1+\epsilon)$, where $|\epsilon|\le\varepsilon=2^{-p}$
and $p$ is the number of bits in the significand of a floating-point
value.  In the LAPACK's terms,
$\varepsilon_{\text{\texttt{x}}}=\text{\texttt{xLAMCH('E')}}$, where
$\text{\texttt{x}}=\text{\texttt{S}}$ or \texttt{D}.  Thus,
$p_{\text{\texttt{x}}}=24$ or $53$, and
$\varepsilon_{\text{\texttt{x}}}=2^{-24}$ or $2^{-53}$ for single
(\texttt{S}) or double (\texttt{D}) precision, respectively.  The
gradual underflow mode, allowing subnormal inputs and outputs, has to
be enabled (e.g., on Intel-compatible architectures the
Denormals-Are-Zero and Flush-To-Zero modes have to be turned off).
Trapping on floating-point exceptions has to be disabled (what is the
default non-stop handling from~\cite[Sect.~7.1]{IEEE-754-2019}).

Possible discretization errors in input data are ignored.
Input matrices in floating-point are thus considered exact, and are,
for simplicity, required to have finite elements.

The Fused Multiply-and-Add (FMA) function,
$\fma(a,b,c)=\fl(a\cdot b+c)$, is required.  Conceptually, the exact
value of $a\cdot b+c$ is correctly rounded.  Also, the hypotenuse
function, $\hypot(a,b)=\fl(\sqrt{a^2+b^2})$, is assumed to be
correctly rounded (unless stated otherwise), as recommended by the
standard~\cite[Sect.~9.2]{IEEE-754-2019}, but
unlike\footnote{\url{https://members.loria.fr/PZimmermann/papers/accuracy.pdf}}
many current implementations of the routines \texttt{hypotf} and
\texttt{hypot}.  Such a function (see also~\cite{Borges-20}) never
overflows when the rounded result should be finite, it is zero if and
only if $|a|=|b|=0$, and is symmetric and monotone.  The CORE-MATH
library~\cite{Sibidanov-et-al-22} provides an open-source
implementation\footnote{\url{https://core-math.gitlabpages.inria.fr}}
of some of the optional correctly rounded single and double precision
mathematical C functions (e.g., \texttt{cr\_hypotf} and
\texttt{cr\_hypot}).

Radix two is assumed for floating-point values.  Scaling of a value
$x$ by $2^s$ where $s\in\mathbb{Z}$ is exact if the result is normal.
Only non-normal results can lose precision.  Let, for $x=\pm 0$,
$e_x=0$ and $f_x=0$, and for a finite non-zero $x$ let the exponent be
$e_x=\lfloor\lg|x|\rfloor$ and the ``mantissa'' $1\le|f_x|<2$, such
that $x=\fl(2^{e_x}f_x)$.  Also, let $f_x=x$ for $x=\pm\infty$, while
$e_x=0$.  Note that $f_x$ is normal even for subnormal $x$.  Keep in
mind that the \texttt{frexp} routine represents a finite non-zero $x$
with $e_x'=e_x^{}+1$ and $f_x'=f_x^{}/2$.

Let $\mu$ be the smallest and $\nu$ the largest positive finite normal
value.  Then, in the notation just introduced, $e_{\mu}=\lg\mu=-126$
or $-1022$, and $e_{\nu}=\lfloor\lg\nu\rfloor=127$ or $1023$, for
single or double precision.  Lemma~\ref{l:nu} can now be stated using
this notation.

\begin{lemma}\label{l:nu}
  Assume that $e_{\nu}-p\ge 1$, with rounding to nearest.  Then,
  \begin{equation}
    \fl(\nu+1)=\nu=\hypot(\nu,1).
    \label{e:nu}
  \end{equation}
\end{lemma}
\begin{proof}
  By the assumption, $\nu\ge 2^{p+1}+2^p+\cdots+2^2$, since
  $\nu=2^{e_{\nu}}(1.1\cdots 1)_2$, with $p$ ones.  The bit in the
  last place thus represents a value of at least four.  Adding one to
  $\nu$ would require rounding of the exact value
  $\nu+1=2^{e_{\nu}}\cdot(1.1\cdots 10\cdots 01)_2$ to $p$ bits of
  significand.  The number of zeros is $e_{\nu}-p\ge 1$.  Rounding to
  nearest in such a case is equivalent to truncating the trailing
  $e_{\nu}-p+1$ bits, starting from the leftmost zero, giving the
  result $\fl(\nu+1)=2^{e_{\nu}}(1.1\cdots 1)_2=\nu$.  This proves the
  first equality in~\eqref{e:nu}.

  For the second equality in~\eqref{e:nu}, note that
  $(\nu+1)^2=\nu^2+1+2\nu>\nu^2+1>\nu^2$ since $\nu>0$.  By taking the
  square roots, it follows that $\nu+1>\sqrt{\nu^2+1}>\nu$, and
  therefore
  \begin{displaymath}
    \nu=\fl(\nu+1)\ge\fl(\sqrt{\nu^2+1})=\hypot(\nu,1)\ge\hypot(\nu,0)=\fl(\nu)=\nu,
  \end{displaymath}
  since $\fl$ and $\hypot$ are monotone operations in all arguments.
\end{proof}

The claims of Lemma~\ref{l:nu} and its following corollaries were used
and their proofs partially sketched in~\cite{Novakovic-24}, e.g.  They
are expanded and clarified here for completeness.

An underlined \underline{expression} denotes a computed floating-point
approximation of the exact value of that expression.  Assuming
$0\le|\phi|\le\pi/4$, and with $\tan(2\phi)$ and
$\underline{\tan(2\phi)}$ given, where the latter may or may not be
equal to $\fl(\tan(2\phi))$, $\tan\phi$ and $\underline{\tan\phi}$ are
\begin{equation}
  \tan\phi=\frac{\tan(2\phi)}{1+\sqrt{\tan^2(2\phi)+1}},\qquad
  \underline{\tan\phi}=\fl\left(\frac{\underline{\tan(2\phi)}}{\fl(1+\hypot(\underline{\tan(2\phi)},1))}\right),
  \label{e:tf}
\end{equation}
if $\tan(2\phi)$ and $\underline{\tan(2\phi)}$ are finite, or
$\sign(\tan(2\phi))$ and $\sign(\underline{\tan(2\phi)})$ otherwise.

\begin{corollary}\label{c:tf}
  Let $\tan(2\phi)$ be given, such that $\underline{\tan(2\phi)}$ is
  finite.  Then, under the assumptions of Lemma~\ref{l:nu}, for
  $\underline{\tan\phi}$ from~\eqref{e:tf} holds
  $0\le|\underline{\tan\phi}|\le 1$.
\end{corollary}
\begin{proof}
  For $|\underline{\tan(2\phi)}|=\nu$, due to Lemma~\ref{l:nu},
  $\hypot(\underline{\tan(2\phi)},1)=\nu$, and so the denominator
  in~\eqref{e:tf} is $\fl(1+\nu)=\nu$.  Note that the numerator is
  always at most as large in magnitude as the denominator.  Thus,
  $0\le|\underline{\tan\phi}|\le 1$, what had to be proven.
\end{proof}

\begin{corollary}\label{c:sec}
  Let $\tan{\phi}$ be given, for $|\phi|\le\pi/2$.  Then,
  $\sec{\phi}=1/\cos{\phi}$ can be approximated as
  $\underline{\sec{\phi}}=\hypot(\underline{\tan{\phi}},1)$.  If
  $\tan{\phi}=\underline{\tan{\phi}}$, then
  $\underline{\sec{\phi}}=\fl(\sec{\phi})$.  When the assumptions of
  Lemma~\ref{l:nu} hold and $\underline{\tan{\phi}}$ is finite, so is
  $\underline{\sec{\phi}}$.
\end{corollary}
\begin{proof}
  The approximation relation follows from the definition of $\hypot$
  and from $\sec\phi=\sqrt{\tan^2\phi+1}$, while its finiteness for a
  finite $\underline{\tan\phi}$ follows from Lemma~\ref{l:nu}, since
  $|\underline{\tan\phi}|\le\nu$ implies
  $\hypot(\underline{\tan\phi},1)\le\nu$.
\end{proof}

For any $w\in\mathbb{R}$, let $\mathbf{w}=(e_w,f_w)=2^{e_w}f_w$ and
$\fl(\mathbf{w})=\fl(2^{e_w}f_w)\approx w$.  Even $w$ such that
$|w|>\nu$ or $0<|w|<\check{\mu}$, where $\check{\mu}$ is the smallest
positive non-zero floating-point value, can be represented with a
finite $e_w$ and a normalized $f_w$, though $\fl(\mathbf{w})$ is not
finite or non-zero, respectively.  The closest double precision
approximation of $\mathbf{w}$ is
$\underline{w}=\text{\texttt{scalbn($f_w,e_w$)}}$, with a possible
underflow or overflow, and similarly for single precision (using
\texttt{scalbnf}).  A similar definition could be made with $e_w'$ and
$f_w'$ instead.

An overflow-avoiding addition and an underflow-avoiding subtraction of
positive finite values $x$ and $y$, resulting in such
exponent-``mantissa'' pairs, can be defined as
\begin{equation}
  x\oplus y=\begin{cases}
  (e_z,f_z),&\text{if $z=\fl(x+y)\le\nu$},\\
  (e_z+1,f_z),&\text{otherwise, with $z=\fl(2^{-1}x+2^{-1}y)$},
  \end{cases}
  \label{e:plus}
\end{equation}
and, assuming $x\ne y$ (for $x=y$ let $x\ominus y=(0,0)$ directly),
\begin{equation}
  x\ominus y=\begin{cases}
  (e_z,f_z),&\text{if $|z|\ge\mu$, with $z=\fl(x-y)$},\\
  (e_z-c,f_z),&\text{otherwise, with $z=\fl(2^c x - 2^c y)$},
  \end{cases}
  \label{e:minus}
\end{equation}
where $c=e_{\mu}+p-1-\min\{e_x,e_y\}$.  In~\eqref{e:plus}, $z\le\nu$
in both cases, since $\fl(\nu/2+\nu/2)=\nu$.

\begin{lemma}\label{l:minus}
  If $e_{\nu}\ge e_{\mu}+2p-1$, there exists $c>0$ such that
  $\mu\le|z|\le\nu$ in~\eqref{e:minus}.
\end{lemma}
\begin{proof}
  Assume $x>y$ (else, swap $x$ and $y$, and change the sign of $z$).
  Then, $e_y\le e_x$ and
  $y=2^{e_y}(1.\mathtt{y}_{-1}\cdots\mathtt{y}_{1-p})_2$.  The
  rightmost bit $\mathtt{y}_{1-p}$ multiplies the value
  $w=2^{e_y+1-p}$.

  If $e_x-e_y\ge p+1$ then $x$ is normal and, due to rounding to
  nearest, $\fl(x-y)=x$.  Therefore, assume that $e_x-e_y\le p$.  If
  $w\ge\mu=2^{e_{\mu}}$, then $x-y\ge\mu$ as well, since $x-y\ge w$.
  Thus, $\fl(x-y)\ge w\ge\mu$, so assume that $w<\mu$, i.e.,
  $e_y+1-p<e_{\mu}$.

  It now suffices to upscale $x$ and $y$ to $x''=2^c x$ and
  $y''=2^c y$, for some $c\in\mathbb{N}$, to ensure
  $e_y''=e_y^{}+c\ge e_{\mu}^{}+p-1$.  Any
  $c\ge e_{\mu}^{}+p-1-e_y^{}$ that will not overflow $x''$ will do,
  so the smallest one is chosen.  Note that
  $e_x''=e_x^{}+c=e_x^{}+e_{\mu}^{}+p-e_y^{}-1$.  Since
  $e_x^{}-e_y^{}\le p$, by the Lemma's assumption it holds
  $e_x''\le e_{\mu}^{}+2p-1\le e_{\nu}$.
\end{proof}

Several arithmetic operations on $(e,f)$-pairs can be defined (see
also~\cite{Novakovic-23}), such as 
\begin{equation}
  \begin{gathered}
    |\mathbf{x}|=(e_x,|f_x|),\qquad-\mathbf{x}=(e_x,-f_x),\qquad 2^{\varsigma}\odot\mathbf{x}=(\varsigma+e_x,f_x),\\
    1\oslash\mathbf{y}=(e_z-e_y,f_z),\quad z=\fl(1/f_y),
  \end{gathered}
  \label{e:unary}
\end{equation}
which are unary operations.  The binary multiplication and division
are defined as
\begin{equation}
  \begin{gathered}
    \mathbf{x}\odot\mathbf{y}=(e_x+e_y+e_z,f_z),\quad
    z=\fl(f_x\cdot f_y),\\
    \mathbf{x}\oslash\mathbf{y}=(e_x-e_y+e_z,f_z),\quad
    z=\fl(f_x/f_y),
  \end{gathered}
  \label{e:binary}
\end{equation}
and the relation $\prec$, that compares the represented values in the
$<$ sense, is given as
\begin{equation}
  \begin{aligned}
    \mathbf{x}\prec\mathbf{y}\iff&(\sign(f_x)<\sign(f_y))\\
    \vee&((\sign(f_x)=\sign(f_y))\wedge((e_x<e_y)\vee((e_x=e_y)\wedge(f_x<f_y)))).
  \end{aligned}
  \label{e:prec}
\end{equation}

Let, for any $G$ of order $n$, where \texttt{INT\_MIN} is the smallest
representable integer,
\begin{equation}
  e_G = \max_{1\le i,j\le n}e_{ij},\quad
  e_{ij}=\max\{\left\lfloor\lg g_{ij}\right\rfloor,\text{\texttt{INT\_MIN}}\}.
  \label{e:eG}
\end{equation}
A prescaling of $G$ as $\underline{G'}=2^s G$, that avoids overflows,
and underflows if possible, in the course of computing the SVD of
$\underline{G'}$ (and thus of
$\underline{G}\approx 2^{-s}\underline{G'}$), is defined by $s$ such
that
\begin{equation}
  e_G=\text{\texttt{INT\_MIN}}\implies
  s=0,\qquad
  e_G>\text{\texttt{INT\_MIN}}\implies
  s=e_{\nu}-e_G-\mathfrak{s},\quad
  \mathfrak{s}\ge 0,
  \label{e:s}
\end{equation}
where $\mathfrak{s}=0$ for $n=1$.  For $n=2$, $\mathfrak{s}$ is chosen
such that certain intermediate results while computing the SVD of
$\underline{G'}$ cannot overflow, as explained
in~\cite{Novakovic-20,Novakovic-Singer-22} and Section~\ref{s:3}, but
the final singular values are represented in the $(e,f)$ form, and are
immune from overflow and underflow as long as they are not converted
to simple floating-point values.  If $s\ge 0$, the result of such a
prescaling is exact.  Otherwise, some elements of $\underline{G'}$
might be computed inexactly due to underflow.  If for the elements of
$G$ holds
\begin{equation}
  g_{ij}\ne 0\implies\mu\le|g_{ij}|\le\nu/2^{\mathfrak{s}},\quad
  1\le i,j\le n,
  \label{e:safe}
\end{equation}
then $s\ge 0$, and $g_{ij}'=0$ or $\mu\le|g_{ij}'|\le\nu$, i.e., the
elements of $G'$ are zero or normal.
\section{The SVD of general matrices of order two}\label{s:3}
This section presents a Kogbetliantz-like procedure for computing the
singular values of $G$ when $n=2$, and the matrices of the left ($U$)
and the right ($V$) singular vectors.

In general, $U$ is a product of permutations (denoted by $P$ with
subscripts and including $I$), sign matrices (denoted by $S$ with
subscripts) with each diagonal element being either $1$ or $-1$ while
the rest are zeros, and plane rotations by the angles $\vartheta$ and
$\varphi$.  If $U_{\vartheta}$ is not generated, the notation changes
from $U_{\varphi}$ to $U_{\phi}$.  Likewise, $V$ is a product of
permutations, a sign matrix, and a plane rotation by the angle $\psi$,
where
\begin{equation}
  U_{\vartheta}=\begin{bmatrix}
  \cos{\vartheta}&-\sin{\vartheta}\\
  \sin{\vartheta}&\hphantom{-}\cos{\vartheta}
  \end{bmatrix},\qquad
  U_{\varphi}=\begin{bmatrix}
  \cos{\varphi}&-\sin{\varphi}\\
  \sin{\varphi}&\hphantom{-}\cos{\varphi}
  \end{bmatrix},\qquad
  V_{\psi}=\begin{bmatrix}
  \cos{\psi}&-\sin{\psi}\\
  \sin{\psi}&\hphantom{-}\cos{\psi}
  \end{bmatrix}.
  \label{e:tfp}
\end{equation}

Depending on its pattern of zeros, a matrix of order two falls into
one of the 16 types $\mathfrak{t}$ shown in~\eqref{e:ts}, where
$\circ=0$ and $\bullet\ne 0$.  Some types are permutationally
equivalent to others, what is denoted by
$\mathfrak{t}_1\cong\mathfrak{t}_2$, and means that a
$\mathfrak{t}_1$-matrix can be pre-multiplied and/or post-multiplied
by permutations to be transformed into a $\mathfrak{t}_2$-matrix, and
vice versa, keeping the number of zeros intact.  Each
$\mathfrak{t}\ne 0$ has its associated scale type $\mathfrak{s}$.
\begin{equation}
  \begin{gathered}
    \begin{gathered}
      \text{\small 0}\\[-3pt]
      \begin{bmatrix}
        \circ & \circ\\[-1pt]
        \circ & \circ
      \end{bmatrix}
    \end{gathered}
    \begin{gathered}
      \text{\small 1}\\[-3pt]
      \begin{bmatrix}
        \bullet & \circ\\[-1pt]
        \circ & \circ
      \end{bmatrix}
    \end{gathered}
    \begin{gathered}
      \text{\small 2}\\[-3pt]
      \begin{bmatrix}
        \circ & \circ\\[-1pt]
        \bullet & \circ
      \end{bmatrix}
    \end{gathered}
    \begin{gathered}
      \text{\small 3}\\[-3pt]
      \begin{bmatrix}
        \bullet & \circ\\[-1pt]
        \bullet & \circ
      \end{bmatrix}
    \end{gathered}
    \begin{gathered}
      \text{\small 4}\\[-3pt]
      \begin{bmatrix}
        \circ & \bullet\\[-1pt]
        \circ & \circ
      \end{bmatrix}
    \end{gathered}
    \begin{gathered}
      \text{\small 5}\\[-3pt]
      \begin{bmatrix}
        \bullet & \bullet\\[-1pt]
        \circ & \circ
      \end{bmatrix}
    \end{gathered}
    \begin{gathered}
      \text{\small 6}\\[-3pt]
      \begin{bmatrix}
        \circ & \bullet\\[-1pt]
        \bullet & \circ
      \end{bmatrix}
    \end{gathered}
    \begin{gathered}
      \text{\small 7}\\[-3pt]
      \begin{bmatrix}
        \bullet & \bullet\\[-1pt]
        \bullet & \circ
      \end{bmatrix}
    \end{gathered}\\
    \begin{gathered}
      \begin{bmatrix}
        \circ & \circ\\[-1pt]
        \circ & \bullet
      \end{bmatrix}\\[-3pt]
      \text{\small 8}
    \end{gathered}
    \begin{gathered}
      \begin{bmatrix}
        \bullet & \circ\\[-1pt]
        \circ & \bullet
      \end{bmatrix}\\[-3pt]
      \text{\small 9}
    \end{gathered}
    \begin{gathered}
      \begin{bmatrix}
        \circ & \circ\\[-1pt]
        \bullet & \bullet
      \end{bmatrix}\\[-3pt]
      \text{\small 10}
    \end{gathered}
    \begin{gathered}
      \begin{bmatrix}
        \bullet & \circ\\[-1pt]
        \bullet & \bullet
      \end{bmatrix}\\[-3pt]
      \text{\small 11}
    \end{gathered}
    \begin{gathered}
      \begin{bmatrix}
        \circ & \bullet\\[-1pt]
        \circ & \bullet
      \end{bmatrix}\\[-3pt]
      \text{\small 12}
    \end{gathered}
    \begin{gathered}
      \begin{bmatrix}
        \bullet & \bullet\\[-1pt]
        \circ & \bullet
      \end{bmatrix}\\[-3pt]
      \text{\small 13}
    \end{gathered}
    \begin{gathered}
      \begin{bmatrix}
        \circ & \bullet\\[-1pt]
        \bullet & \bullet
      \end{bmatrix}\\[-3pt]
      \text{\small 14}
    \end{gathered}
    \begin{gathered}
      \begin{bmatrix}
        \bullet & \bullet\\[-1pt]
        \bullet & \bullet
      \end{bmatrix}\\[-3pt]
      \text{\small 15}
    \end{gathered}
  \end{gathered}\quad
  \begin{gathered}
    \mathfrak{t}\\
    0, 1, 2, 4, 6, 8\cong\mathbf{9}\\
    12\cong\mathbf{3};\quad 10\cong\mathbf{5}\\
    7, 11, 14\cong\mathbf{13}\\
    \quad\mathbf{15}
  \end{gathered}\quad
  \begin{gathered}
    \mathfrak{s}\\
    0\\
    1\\
    1\\
    2
  \end{gathered}
  \label{e:ts}
\end{equation}

For $\mathfrak{s}=0$, there is one equivalence class of matrix types,
represented by $\mathfrak{t}=9$.  For $\mathfrak{s}=1$, there are
three classes, represented by $\mathfrak{t}=3$, $\mathfrak{t}=5$, and
$\mathfrak{t}=13$, while for $\mathfrak{s}=2$ there is one class,
$\mathfrak{t}=15$.  The SVD computation for the first three classes is
straightforward, while for the fourth and the fifth class is more
involved.  A matrix of any type, except $\mathfrak{t}=15$, can be
permuted into an upper triangular one.  If a matrix so obtained is
well scaled, its SVD can alternatively be computed by \texttt{xLASV2}.
However, \texttt{xLASV2} does not accept general matrices (i.e.,
$\mathfrak{t}=15$), unlike the proposed method, which is a
modification of the ``trigonometric'' case
from~\cite{Novakovic-Singer-22}, i.e., of the case where the sign
matrix $J=I$.  The proposed method consists of the following three
phases:
\begin{enumerate}
\item For $G$ determine $\mathfrak{t}$, $\mathfrak{s}$, and $s$ to
  obtain $\underline{G'}$.  Handle the simple cases of $\mathfrak{t}$
  separately.
\item If $\mathfrak{t}\cong 13$ or $\mathfrak{t}=15$, factorize
  $\underline{G'}$ as $U_+^{}RV_+^{}$, such that $U_+^{}$ and
  $V_+^{}$ are orthogonal, and $R$ is upper triangular, with
  $\min\{r_{11}^{},r_{12}^{},r_{22}^{}\}>0$ and all $r_{ij}$ finite,
  $1\le i\le j\le 2$.
\item From the SVD of $\underline{R}$ assemble the SVD of
  $\underline{G'}$.  Optionally backscale $\underline{\Sigma'}$ by
  $2^{-s}$.
\end{enumerate}
The phases 1, 2, and 3 are described in Sections~\ref{ss:3.1},
\ref{ss:3.2}, and~\ref{ss:3.3}, respectively.
\subsection{Prescaling of the matrix and the simple cases ($\mathfrak{t}\cong 3,5,9$)}\label{ss:3.1}
Matrices with $\mathfrak{t}\cong 9$ do not have to be scaled, but only
permuted into the $\mathfrak{t}=0$, $\mathfrak{t}=1$, or
$\mathfrak{t}=9$ (where the first diagonal element is not smaller by
magnitude than the second one) form, according to their number of
non-zeros, with at most one permutation from the left and at most one
from the right hand side.  Then, the rows of $P_U^T G P_V^{}$ are
multiplied by the signs of their diagonal elements, to obtain 
$\sigma_1=|g_{11}|$ and $\sigma_2=|g_{22}|$, while $U=P_US$ and
$V=P_V$.  The error-free SVD computation is thus completed.

Note that the signs might have been taken out of the columns instead
of the rows, and the sign matrix $S$ would have then be incorporated
into $V$ instead.  The structure of the left and the right singular
vector matrices is therefore not uniquely determined.

Be aware that $\mathfrak{t}$ determined before the prescaling (to
compute $\mathfrak{s}$ and $s$) may differ from $\mathfrak{t}'$ that
would be found afterwards.  If, e.g., $\mathfrak{t}\ncong 9$ and $G$
contains, among others, $\nu$ and $\check{\mu}$ as elements, the
element(s) $\check{\mu}$ will vanish after the prescaling since $s<0$
(from~\eqref{e:s}, due to $\mathfrak{s}\ge 1$), so
$\mathfrak{t}'<\mathfrak{t}$ and the zero pattern of
$\underline{G'}$ has to be re-examined.

\looseness=-1
A $\mathfrak{t}\cong 3$ or $\mathfrak{t}\cong 5$ matrix is scaled by
$2^s$.  The columns (resp., rows) of a $\mathfrak{t}'=12$ (resp.,
$\mathfrak{t}'=10$) matrix are swapped, to bring it to the
$\mathfrak{t}''=3$ (resp., $\mathfrak{t}''=5$) form.  Then, the
non-zero elements are made positive by multiplying the rows (resp.,
columns) by their signs. Next, the rows (resp., columns) are swapped
if required to make the upper left element largest by magnitude.  The
sign-extracting and magnitude-ordering operations may be swapped or
combined.  The resulting matrix $G''$ undergoes the QR (resp., RQ)
factorization, by a single Givens rotation $U_{\theta}^T$ (resp.,
$V_{\theta}^{}$), determined by $\tan\theta$ (consequently, by
$\cos\theta$ and $\sin\theta$) as in~\eqref{e:tfp}, with $\theta$
substituted for $\vartheta$ (resp., $\psi$), where
\begin{equation}
  \mathfrak{t}''=3\implies
  \tan\theta=g_{21}''/g_{11}'',\qquad
  \mathfrak{t}''=5\implies
  \tan\theta=g_{12}''/g_{11}''.
  \label{e:theta}
\end{equation}

\looseness=-1
By construction, $0<\tan\theta\le 1$ and
$0\le\underline{\tan\theta}\le 1$.  The upper left element is not
transformed, but explicitly set to hold the Frobenius norm of the
whole non-zero column (resp., row), as
$\underline{g_{11}'''}=\hypot(\underline{g_{11}''},\underline{g_{21}''})$
(resp.,
$\underline{g_{11}'''}=\hypot(\underline{g_{11}''},\underline{g_{12}''})$),
while the other non-zero element is zeroed out.  Thus, to avoid
overflow of $\underline{g_{11}'''}$ it is sufficient to ensure that
$\underline{g_{11}''}\ll\nu/\sqrt{2}$, what $\mathfrak{s}=1$ achieves.
The SVD is given by $U=S_U^{}P_U^{}U_{\theta}^{}$ and $V=P_V^{}$ for
$\mathfrak{t}'\cong 3$, and by $U=P_U^{}$ and
$V=S_V^{}P_V^{}V_{\theta}^{}$ for $\mathfrak{t}'\cong 5$.  The scaled
singular values are $\sigma_1'=g_{11}'''$ and $\sigma_2'=0$ in both
cases, and $\underline{\sigma_1'}$ cannot overflow
($\underline{\sigma_1^{}}=2^{-s}\underline{\sigma_1'}$ can).

If no inexact underflow occurs while scaling $G$ to $G'$, then
$\underline{\sigma_1'}=\sigma_1'(1+\epsilon_1')$, where
$|\epsilon_1'|\le\varepsilon$.  With the same assumption,
$\underline{\tan\theta}\ge\mu$ implies
$\underline{\tan\theta}=\tan\theta(1+\epsilon_{\theta}^{})$, where
$|\epsilon_{\theta}^{}|\le\varepsilon$.  The resulting Givens rotation
can be represented and applied as one of
\begin{equation}
  \underline{U_{\theta}^T}=\begin{bmatrix}
  1 & \underline{\tan\theta}\\
  -\underline{\tan\theta} & 1
  \end{bmatrix}/\underline{\sec\theta},\qquad
  \underline{V_{\theta}^{}}=\begin{bmatrix}
  1 & -\underline{\tan\theta}\\
  \underline{\tan\theta} & 1
  \end{bmatrix}/\underline{\sec\theta},
  \label{e:Givens}
\end{equation}
what avoids computing $\underline{\cos\theta}$ and
$\underline{\sin\theta}$ explicitly.  Lemma~\ref{l:sectheta} bounds the
error in $\underline{\sec\theta}$.

\begin{lemma}\label{l:sectheta}
  Let $\underline{\sec\theta}$ from~\eqref{e:Givens} be computed as
  $\hypot(\underline{\tan\theta},1)$ for $\underline{\tan\theta}\ge\mu$.
  Then,
  \begin{equation}
    \underline{\sec\theta}=\delta_{\theta}'\sec\theta,\quad
    \sqrt{((1-\varepsilon)^2+1)/2}(1-\varepsilon)\le\delta_{\theta}'\le\sqrt{((1+\varepsilon)^2+1)/2}(1+\varepsilon).
    \label{e:sectheta}
  \end{equation}
\end{lemma}
\begin{proof}
  Let $\delta_{\theta}^{}=(1+\epsilon_{\theta}^{})$.  Then
  $\underline{\tan\theta}\ge\mu$ implies
  $(\underline{\tan\theta})^2=\delta_{\theta}^2\tan^2\theta$, and
  \begin{displaymath}
    1-\varepsilon=\delta_{\theta}^-\le\delta_{\theta}^{}\le\delta_{\theta}^+=1+\varepsilon.
  \end{displaymath}
  Express
  $(\underline{\tan\theta})^2+1=\delta_{\theta}^2\tan^2\theta+1$ as
  $(\tan^2\theta+1)(1+\epsilon_{\theta}')$, from which it follows
  \begin{displaymath}
    \epsilon_{\theta}'=\frac{\tan^2\theta}{\tan^2\theta+1}(\delta_{\theta}^2-1),\qquad
    0<\theta\le\frac{\pi}{4}\implies 0<\frac{\tan^2\theta}{\tan^2\theta+1}\le\frac{1}{2}.
  \end{displaymath}
  By adding unity to both sides of the equation for
  $\epsilon_{\theta}'$ and taking the maximal value of the first
  factor on its right hand side, while accounting for the bounds of
  $\delta_{\theta}^{}$, it holds
  \begin{displaymath}
    ((\delta_{\theta}^-)^2+1)/2\le 1+\epsilon_{\theta}'\le((\delta_{\theta}^+)^2+1)/2.
  \end{displaymath}
  Since
  $\underline{\sec\theta}=\hypot(\underline{\tan\theta},1)=\sqrt{(\underline{\tan\theta})^2+1}(1+\epsilon_{\sqrt{\null}}^{})=\sqrt{(\tan^2\theta+1)(1+\epsilon_{\theta}')}(1+\epsilon_{\sqrt{\null}}^{})$,
  where $|\epsilon_{\sqrt{\null}}|\le\varepsilon$, factorizing the
  last square root into a product of square roots gives
  \begin{displaymath}
    \underline{\sec\theta}=\sec\theta\sqrt{1+\epsilon_{\theta}'}(1+\epsilon_{\sqrt{\null}}^{}).
  \end{displaymath}
  The proof is concluded by denoting the error factor on the right
  hand side by $\delta_{\theta}'$.
\end{proof}

This proof and the following one use several techniques
from~\cite[Theorem~1]{Novakovic-24}.  Due to the structure of a matrix
that $U_{\theta}^T$ or $V_{\theta}^{}$ is applied to, containing in
each row and column one zero and one $\pm 1$,  it follows that $U$ or
$V$ have in each row and column $\pm\underline{\cos\theta}$ and
$\pm\underline{\sin\theta}$, computed implicitly, for which
Lemma~\ref{l:sincostheta} gives error bounds.

\begin{lemma}\label{l:sincostheta}
  Let $\underline{\cos\theta}$ and $\underline{\sin\theta}$ result
  from applying~\eqref{e:Givens} with $\underline{\tan\theta}\ge\mu$.
  Then,
  \begin{equation}
    \underline{\cos\theta}=\delta_{\theta}''\cos\theta,\quad
    \delta_{\theta}''=(1+\epsilon_/^{})/\delta_{\theta}',\qquad
    \underline{\sin\theta}=\delta_{\theta}'''\sin\theta,\quad
    \delta_{\theta}'''=(1+\epsilon_/')\delta_{\theta}^{}/\delta_{\theta}',
    \label{e:sincostheta}
  \end{equation}
  where $\max\{|\epsilon_/^{}|,|\epsilon_/'|\}\le\varepsilon$ and
  $\delta_{\theta}''$ and $\delta_{\theta}'''$ can be bound below and
  above in the terms of $\varepsilon$ only.  Let
  $\delta_{\theta}^{\prime-}$ and $\delta_{\theta}^{\prime+}$ be the
  lower and the upper bounds for $\delta_{\theta}'$
  from~\eqref{e:sectheta}.  Then,
  \begin{equation}
    (1-\varepsilon)/\delta_{\theta}^{\prime+}\le\delta_{\theta}''\le(1+\varepsilon)/\delta_{\theta}^{\prime-},\qquad
    (1-\varepsilon)\delta_{\theta}^-/\delta_{\theta}^{\prime+}\le\delta_{\theta}'''\le(1+\varepsilon)\delta_{\theta}^+/\delta_{\theta}^{\prime-}.
    \label{e:sincosthetabound}
  \end{equation}
\end{lemma}
\begin{proof}
  The claims follow from~\eqref{e:sectheta},
  $\underline{\cos\theta}=\fl(1/\underline{\sec\theta})$, and
  $\underline{\sin\theta}=\fl(\underline{\tan\theta}/\underline{\sec\theta})$.
\end{proof}

If, in~\eqref{e:theta}, $\underline{\tan\theta}<\mu$, then
$\underline{\sec\theta}=1$ since
$1\le\underline{\sec\theta}\le\fl(\sqrt{1+\mu^2})\le\fl(1+\mu)=1$, and the
relative error in $\underline{\sec\theta}$ is below $\varepsilon$ for
any standard floating-point datatype.  Thus, even though
$\underline{\tan\theta}$ can be relatively inaccurate,
\eqref{e:sectheta} holds for all $\underline{\sec\theta}$.
Also, $\underline{\cos\theta}$ is always relatively
accurate, but $\underline{\sin\theta}$ might not be if
$\underline{\tan\theta}<\mu$, when
$\underline{\sin\theta}=\underline{\tan\theta}$.
\subsection{A (pivoted) $URV$ factorization of order two ($\mathfrak{t}'\cong 13,15$)}\label{ss:3.2}
If, after the prescaling, $\mathfrak{t}'\cong 13$ or
$\mathfrak{t}'=15$, $\underline{G'}$ is transformed into an upper
triangular matrix $R$ with all elements non-negative, i.e., a special
URV factorization of $\underline{G'}$ is computed.
Section~\ref{sss:3.2.1} deals with the $\mathfrak{t}'\cong 13$, and
Section~\ref{sss:3.2.2} with the $\mathfrak{t}'=15$ case.
\subsubsection{An error-free transformation from $\mathfrak{t}'\cong 13$ to $\mathfrak{t}''=13$ form}\label{sss:3.2.1}
A triangular or anti-triangular matrix is first permuted into an upper
triangular one, $G''$.  Its first row is then multiplied by the sign
of $g_{11}''$.  This might change the sign of $g_{12}''$.  The second
column is multiplied by its new sign, what might change the
sign of $g_{22}''$.  Its new sign then multiplies the second row, what
completes the construction of $R$.

The transformations $U_+^T$ and $V_+^{}$, such that
$R=U_+^T G' V_+^{}$, can be expressed as
$U_+^{}=P_U^{}S_{11}^{}S_{22}^{}=\underline{U_+^{}}$ and
$V_+^{}=P_V^{}S_{12}^{}=\underline{V_+^{}}$, and are exact, as well as
$\underline{R}$ if $\underline{G'}$ is.
\subsubsection{A fully pivoted URV when $\mathfrak{t}'=15$}\label{sss:3.2.2}
In all previous cases, a sequence of error-free transformations would
bring $G'$ into an upper triangular $G''$, of which \texttt{xLASV2}
can compute the SVD\@.  However, a matrix without zeros either has to
be preprocessed into such a form, in the spirit
of~\cite{Novakovic-20,Novakovic-Singer-22}, or its SVD has to computed
by more complicated and numerically less stable formulas, that follow
from the annihilation requirement for the off-diagonal matrix elements
as
\begin{equation}
  \frac{\tan(2\phi)}{2}=\frac{g_{11}^{}g_{21}^{}+g_{12}^{}g_{22}^{}}{g_{11}^2+g_{12}^2-g_{21}^2-g_{22}^2},\qquad
  \tan\psi=\frac{g_{12}^{}+g_{22}^{}\tan\phi}{g_{11}^{}+g_{21}^{}\tan\phi}.
  \label{e:trad}
\end{equation}
A sketched derivation of~\eqref{e:trad} can be found in Section~1 of
the supplementary material.

Opting for the first approach, compute the Frobenius norms of the
columns of $G'$, as $w_1^{}$ and $w_2^{}$.  Due to the prescaling,
$\underline{w_1^{}}=\hypot(\underline{g_{11}'},\underline{g_{21}'})$
and
$\underline{w_2^{}}=\hypot(\underline{g_{12}'},\underline{g_{22}'})$
cannot overflow.  If $w_1^{}<w_2^{}$, swap the columns and their norms
(so that $w_1'$ would be the norm of the new first column of
$\widetilde{G}'$, and $w_2'$ the norm of the second one).  Multiply
each row by the sign of its new first element to get $G''$.  Swap the
rows if $g_{11}''<g_{21}''$ to get the fully pivoted $G'''$, while the
norms remain unchanged.  Note that $g_{11}'''\ge g_{21}'''>0$.

Now the QR factorization of $G'''$ is computed as
$U_{\vartheta}^TG'''=R''$.  Then, $r_{21}''=0$, and
\begin{equation}
  r_{11}''=w_1',\quad
  \tan\vartheta=\frac{g_{21}'''}{g_{11}'''},\quad
  r_{12}''=\frac{g_{12}'''+g_{22}'''\tan\vartheta}{\sec\vartheta},\quad
  r_{22}''=\frac{g_{22}'''-g_{12}'''\tan\vartheta}{\sec\vartheta}.
  \label{e:QR15}
\end{equation}
All properties of the functions of $\theta$ from Section~\ref{ss:3.1}
also hold for the functions of $\vartheta$.  The prescaling of $G$
causes the elements of $R''$ to be at most $\nu/(2\sqrt{2})$ in
magnitude.

If $\underline{r_{12}''}=0$, then $\mathfrak{t}''=9$, and if
$\underline{r_{22}''}=0$, then $\mathfrak{t}''=5$.  In either case
$\underline{R''}$ is processed further as in Section~\ref{ss:3.1},
while accounting for the already applied transformations.  Else, the
second column of $R''$ is multiplied by the sign of $r_{12}''$ to
obtain $R'$.  The second row of $R'$ is multiplied by the sign of
$r_{22}'$ to finalize $R$, in which the upper triangle is positive.
It is evident how to construct $U_+^{}$ and
$V_+^{}=\underline{V_+^{}}$ such that $R=U_+^T G' V_+^{}$, since
\begin{equation}
  U_+^T=S_{22}^TU_{\vartheta}^TP_U^TS_1^T,\quad
  S_1^{}=\diag(\sign(\tilde{g}_{11}'),\sign(\tilde{g}_{21}')),\qquad
  V_+^{}=P_V^{}S_{12}^{}.
  \label{e:UV15}
\end{equation}
However, $\underline{U_+^T}$ is \emph{not} explicitly formed, as
explained in Section~\ref{ss:3.3}.  Now $\mathfrak{t}''=13$ for
$\underline{R}$.

Lemma~\ref{l:R} bounds the relative errors in (some of) the elements
of $\underline{R}$, when possible.

\begin{lemma}\label{l:R}
  Assume that no inexact underflow occurs at any stage of the above
  computation, leading from $G$ to $\underline{R}$.  Then,
  $\underline{r_{11}}=\delta_0 r_{11}$, where
  \begin{equation}
    1-\varepsilon=\delta_0^-\le\delta_0^{}\le\delta_0^+=1+\varepsilon.
    \label{e:d0}
  \end{equation}
  If in $\underline{\tan\vartheta}=\delta_{\vartheta}^{}\tan\vartheta$
  holds $\delta_{\vartheta}^{}=1$, then
  $\underline{r_{12}^{}}=\delta_1'r_{12}^{}$ and
  $\underline{r_{22}^{}}=\delta_1''r_{22}^{}$, where
  \begin{equation}
    \frac{(1-\varepsilon)^2}{\sqrt{((1+\varepsilon)^2+1)/2}(1+\varepsilon)}=\delta_1^-\le\delta_1',\delta_1''\le\delta_1^+=\frac{(1+\varepsilon)^2}{\sqrt{((1-\varepsilon)^2+1)/2}(1-\varepsilon)}.
    \label{e:d1}
  \end{equation}
  Else, if $\underline{g_{12}'''}$ and $\underline{g_{22}'''}$ are of
  the same sign, then $\underline{r_{12}^{}}=\delta_2'r_{12}^{}$, and
  if they are of the opposite signs, then
  $\underline{r_{22}^{}}=\delta_2''r_{22}^{}$, where, with
  $1-\varepsilon=\delta_{\vartheta}^-\le\delta_{\vartheta}^{}\le\delta_{\vartheta}^+=1+\varepsilon$
  and $\delta_{\vartheta}^{}\ne 1$,
  \begin{equation}
    \frac{(1-\varepsilon)^3}{\sqrt{((1+\varepsilon)^2+1)/2}(1+\varepsilon)}=\delta_2^-<\delta_2',\delta_2''<\delta_2^+=\frac{(1+\varepsilon)^3}{\sqrt{((1-\varepsilon)^2+1)/2}(1-\varepsilon)}.
    \label{e:d2}
  \end{equation}
\end{lemma}
\begin{proof}
  Eq.~\eqref{e:d0} follows from the correct rounding of $\hypot$
  in the computation of $\underline{w_1'}$.

  To prove~\eqref{e:d1} and~\eqref{e:d2}, solve
  $x\pm y\delta_{\vartheta}\tan\vartheta=(x\pm y\tan\vartheta)(1+\epsilon_{\pm})$
  for $\epsilon_{\pm}$ with $xy\ne 0$.  After expanding and
  rearranging the terms, it follows that
  \begin{equation}
    \epsilon_{\pm}=\frac{\pm y\tan\vartheta}{x\pm y\tan\vartheta}(\delta_{\vartheta}-1),\qquad
    \delta_{\vartheta}=1+\epsilon_{\vartheta},\quad
    |\epsilon_{\vartheta}|\le\varepsilon.
    \label{e:epm}
  \end{equation}
  If $x$ and $y$ are of the same sign and the addition operation is
  chosen, the first factor on the first right hand side
  in~\eqref{e:epm} is above zero and below unity, so
  $|\epsilon_{\pm}|<|\delta_{\vartheta}-1|$ and
  \begin{equation}
    \delta_{\vartheta}^-=\delta_{\pm}^-<\delta_{\pm}^{}<\delta_{\pm}^+=\delta_{\vartheta}^+,\qquad
    \delta_{\pm}^{}=1+\epsilon_{\pm}^{}.
    \label{e:dpm}
  \end{equation}
  The same holds if $x$ and $y$ are of the opposite signs and the
  subtraction is taken instead.  Namely, from~\eqref{e:QR15}, the
  bound~\eqref{e:dpm} holds for $x'=\underline{g_{12}'''}$ and
  $y'=\underline{g_{22}'''}$ of the \emph{same} sign,
  \begin{equation}
    \sign(x'=\underline{g_{12}'''})=\sign(y'=\underline{g_{22}'''})\implies\delta_{\pm}^{}\mapsto\delta_+^{},
    \label{e:dp}
  \end{equation}
  i.e., with $\delta_{\pm}$ denoted as $\delta_+$, and for
  $x''=\underline{g_{22}'''}$ and $y''=\underline{g_{12}'''}$ of the
  \emph{opposite} signs,
  \begin{equation}
    \sign(x''=\underline{g_{22}'''})=-\sign(y''=\underline{g_{12}'''})\implies\delta_{\pm}^{}\mapsto\delta_-^{},
    \label{e:dm}
  \end{equation}
  with $\delta_{\pm}$ becoming $\delta_-$ in this case.  When both
  cases are considered together, $\delta_{\pm}$ is used.

  With $|\epsilon_{\fma}|\le\varepsilon$ and
  $\delta_{\fma}=1+\epsilon_{\fma}$, from the definition of
  $\delta_{\pm}$ it follows that
  \begin{equation}
    \fma(\pm y,\underline{\tan\vartheta},x)=(x\pm y\underline{\tan\vartheta})\delta_{\fma}=(x\pm y\tan\vartheta)\delta_{\fma}\delta_{\pm}.
    \label{e:xpmy}
  \end{equation}
  Due to~\eqref{e:sectheta} and~\eqref{e:xpmy}, with $\vartheta$
  instead of $\theta$ and with $\delta_/''=(1+\epsilon_/'')$ where
  $|\epsilon_/''|\le\varepsilon$, it holds
  \begin{equation}
    \fl\left(\frac{\fma(y',\underline{\tan\vartheta},x')}{\underline{\sec\vartheta}}\right)=\frac{x'+y'\tan\vartheta}{\sec\vartheta}\cdot\frac{\delta_{\fma}^{}\delta_/''}{\delta_{\vartheta}'}\cdot\delta_+^{}=\underline{r_{12}''}\cdot\delta_1'\cdot\delta_+^{}=\underline{r_{12}''}\delta_2',
    \label{e:d1d2p}
  \end{equation}
  in the case~\eqref{e:dp}, while in the case~\eqref{e:dm} it is
  similarly obtained that
  \begin{equation}
    \fl\left(\frac{\fma(-y'',\underline{\tan\vartheta},x'')}{\underline{\sec\vartheta}}\right)=\frac{x''-y''\tan\vartheta}{\sec\vartheta}\cdot\frac{\delta_{\fma}^{}\delta_/''}{\delta_{\vartheta}'}\cdot\delta_-^{}=\underline{r_{22}''}\cdot\delta_1''\cdot\delta_-^{}=\underline{r_{22}''}\delta_2''.
    \label{e:d1d2m}
  \end{equation}
  Now~\eqref{e:d1} follows from bounding $\delta_1'$ and $\delta_1''$,
  standing for $\delta_{\fma}^{}\delta_/''/\delta_{\vartheta}'$
  in~\eqref{e:d1d2p} and~\eqref{e:d1d2m}, respectively, below and
  above using~\eqref{e:sectheta}.  Since
  $\delta_2'=\delta_1'\delta_+^{}$ in~\eqref{e:d1d2p}, and
  $\delta_2''=\delta_1''\delta_-^{}$ in~\eqref{e:d1d2m}, the proof is
  concluded by observing that~\eqref{e:d2} is a consequence
  of~\eqref{e:d1}.
\end{proof}

Lemma~\ref{l:R} thus shows that high relative accuracy of all elements
of $\underline{R}$ is achieved if no underflow has occurred at any
stage, and $\underline{\tan\vartheta}$ has been computed exactly.  If
$\underline{\tan\vartheta}$ is inexact, high relative accuracy is
guaranteed for $\underline{r_{11}}$ and exactly one of
$\underline{r_{12}}$ and $\underline{r_{22}}$.  If it is also desired
for the remaining element, transformed by an essential subtraction and
thus amenable to cancellation, one possibility is to compute it by an
expression equivalent to~\eqref{e:QR15}, but with $\tan\vartheta$
expanded to its definition $g_{21}'''/g_{11}'''$, as in
\begin{equation}
  r_{12}''=(g_{12}'''g_{11}'''+g_{22}'''g_{21}''')/(g_{11}'''\sec\vartheta),\quad
  r_{22}''=(g_{22}'''g_{11}'''-g_{12}'''g_{21}''')/(g_{11}'''\sec\vartheta),
  \label{e:ext}
\end{equation}
after prescaling the numerator and denominator in~\eqref{e:ext} by the
largest power-of-two that avoids overflow of both of them.  A
floating-point primitive of the form $a\cdot b\pm c\cdot d$ with a
single rounding~\cite{Lauter-17} can give the correctly rounded
numerator but it has to be emulated in software on most platforms at
present~\cite{Hubrecht-et-al-24}.  For high relative accuracy of
$\underline{r_{12}''}$ or $\underline{r_{22}''}$ the absence of
inexact underflows is required, except in the prescaling.

Alternatively, the numerators in~\eqref{e:ext} can be calculated by the
Kahan's algorithm for determinants of order
two~\cite{Jeannerod-et-al-13}, but an overflow-avoiding prescaling is
still necessary.  It is thus easier to resort to
computing~\eqref{e:ext} in a wider and more precise datatype as
\begin{equation}
  \begin{aligned}
    \underline{r_{12}''}&=\fl(\fl(\FL(\FL(g_{12}'''g_{11}'''+g_{22}'''g_{21}''')/g_{11}'''))/\underline{\sec\vartheta}),\\
    \underline{r_{22}''}&=\fl(\fl(\FL(\FL(g_{22}'''g_{11}'''-g_{12}'''g_{21}''')/g_{11}'''))/\underline{\sec\vartheta}).
  \end{aligned}
  \label{e:EXT}
\end{equation}

First, in~\eqref{e:EXT} it is assumed that no underflow has occurred so
far, so $\underline{G'''}=G'''$.  Second, a product of two
floating-point values requires $2p$ bits of the significand to be
represented exactly if the factors' significands are encoded with $p$
bits each.  Thus, for single precision, $48$ bits are needed, what is
less than $53$ bits available in double precision.  Similarly, for
double precision, $106$ bits are needed, what is less than $113$ bits
of a quadruple precision significand.  Therefore, every product
in~\eqref{e:EXT} is exact if computed using a more precise standard
datatype \texttt{T}.  The characteristic values of \texttt{T} are the
underflow and overflow thresholds $\mu_{\text{\texttt{T}}}$ and
$\nu_{\text{\texttt{T}}}$, and
$\varepsilon_{\text{\texttt{T}}}=2^{-p_{\text{\texttt{T}}}}$.  Third,
let $\FL(x)$ round an infinitely precise result of $x$ to the nearest
value in \texttt{T}.  Since all addends in~\eqref{e:EXT} are exact and
way above the underflow threshold $\mu_{\text{\texttt{T}}}$ by
magnitude, the rounded result of the addition or the subtraction is
relatively accurate, with the error factor $(1+\epsilon_{\pm}')$,
$|\epsilon_{\pm}'|\le\varepsilon_{\text{\texttt{T}}}^{}$.  This holds
even if the result is zero, but since the transformed matrix would
then be processed according to its new structure, assume that the
result is normal in \texttt{T}.  The ensuing division cannot overflow
nor underflow in \texttt{T}.  Now the quotient rounded by $\FL$ is
rounded again, by $\fl$, back to the working datatype.  This operation
can underflow, as well as the following division by
$\underline{\sec\vartheta}$.  If they do not, the resulting
transformed element is relatively accurate.  This outlines the
proof of Theorem~\ref{t:R}.

\begin{theorem}\label{t:R}
  Assume that no underflow occurs at any stage of the computation
  leading from $G$ to $\underline{R}$.  Then,
  $\underline{r_{11}}=\delta_0 r_{11}$, where for $\delta_0$
  holds~\eqref{e:d0}.  If $\underline{\tan\vartheta}$ is exact, then
  $\underline{r_{12}^{}}=\delta_1'r_{12}^{}$ and
  $\underline{r_{22}^{}}=\delta_1''r_{22}^{}$, where $\delta_1'$ and
  $\delta_1''$ are as in~\eqref{e:d1}.   Else, if
  $\underline{g_{12}'''}$ and $\underline{g_{22}'''}$ are of the same
  sign, then $\underline{r_{12}^{}}=\delta_2'r_{12}^{}$ and
  $\underline{r_{22}^{}}=\delta_3'r_{22}^{}$.  If they are of the
  opposite signs, then $\underline{r_{22}^{}}=\delta_2''r_{22}^{}$ and
  $\underline{r_{12}^{}}=\delta_3''r_{12}^{}$, where $\delta_2'$ and
  $\delta_2''$ are as in~\eqref{e:d2}, while $\delta_3'$ and
  $\delta_3''$ come from evaluating their corresponding matrix
  elements as in~\eqref{e:EXT} and are bounded as
  \begin{equation}
    \frac{(1-\varepsilon_{\text{\texttt{T}}}^{})(1-\varepsilon)^2}{\sqrt{((1+\varepsilon)^2+1)/2}(1+\varepsilon)}=\delta_3^-\le\delta_3',\delta_3''\le\delta_3^+=\frac{(1+\varepsilon_{\text{\texttt{T}}}^{})(1+\varepsilon)^2}{\sqrt{((1-\varepsilon)^2+1)/2}(1-\varepsilon)}.
    \label{e:d3}
  \end{equation}
\end{theorem}
\begin{proof}
  It remains to prove~\eqref{e:d3}, since the other relations follow
  from Lemma~\ref{l:R}.

  Every element of $G'''$ is not above $\nu/4$ and not below $\mu$ in
  magnitude.  Therefore, a difference (in essence) of their exact
  products cannot exceed $\nu^2/16<\nu_{\text{\texttt{T}}}^{}$ in
  magnitude in a standard \texttt{T}.  At least one element is above
  $\nu/8$ in magnitude due to the prescaling, so the said difference,
  if not exactly zero, is above
  $\varepsilon\mu\nu/8\gg\mu_{\text{\texttt{T}}}^{}$ in magnitude.
  Thus, the quotient of this difference and $g_{11}'''$ is above
  $\varepsilon\mu(1-\varepsilon_{\text{\texttt{T}}}^{})/2>\mu_{\text{\texttt{T}}}^{}$
  and, due to the prescaling and pivoting, not above
  $\nu/2\ll\nu_{\text{\texttt{T}}}$ in magnitude.  For~\eqref{e:EXT}
  it therefore holds
  \begin{equation}
    \begin{gathered}
    \begin{aligned}
      \underline{\star_+^{}}&=\FL(\star_+^{})=\star_+^{}(1+\epsilon_+'),\quad
      \star_+^{}=g_{12}'''g_{11}'''+g_{22}'''g_{21}''',\quad
      |\epsilon_+'|\le\varepsilon_{\text{\texttt{T}}}^{},\\
      \underline{\star_-^{}}&=\FL(\star_-^{})=\star_-^{}(1+\epsilon_-'),\quad
      \star_-^{}=g_{22}'''g_{11}'''-g_{12}'''g_{21}''',\quad
      |\epsilon_-'|\le\varepsilon_{\text{\texttt{T}}}^{},
    \end{aligned}\\
    \fl(\FL(\underline{\star_{\pm}^{}}/g_{11}'''))=\fl(\underline{\star_{\pm}^{}}/g_{11}''')=(\star_{\pm}^{}/g_{11}''')(1+\epsilon_{\pm}')(1+\epsilon_/''')=(\star_{\pm}^{}/g_{11}''')\delta_{\pm}',
    \end{gathered}
    \label{e:star}
  \end{equation}
  with $|\epsilon_/'''|\le\varepsilon$, since the doubly-rounded
  quotient $\fl(\FL(\underline{\star_{\pm}^{}}/g_{11}'''))$ is
  correctly rounded to the working precision from the value of
  $\underline{\star_{\pm}^{}}/g_{11}'''$, where $\star_{\pm}^{}$ can
  be replaced by $\star_+^{}$ or $\star_-$.  This double rounding is
  innocuous with the \emph{standard} datatypes, i.e, when rounding
  from exact to double to single and from exact to quadruple to double
  precision, due to~\cite[Theorem~4]{Figueroa-95} and~\cite[Theorem~29
    and Table~II]{Roux-14}, but might not be for specialized,
  non-standard datatypes, for which the proposed method would have to
  be redesigned.

  For the final division by $\underline{\sec\vartheta}$, due
  to~\eqref{e:sectheta} and~\eqref{e:star}, it holds
  \begin{equation}
    \fl\left(\frac{\fl(\FL(\underline{\star_{\pm}^{}}/g_{11}'''))}{\underline{\sec\vartheta}}\right)=\frac{\star_{\pm}^{}}{g_{11}'''\sec\vartheta}\frac{\delta_{\pm}'}{\delta_{\vartheta}'}\delta_/'''',\quad
    \delta_/''''=(1+\epsilon_/''''),\quad|\epsilon_/''''|\le\varepsilon.
    \label{e:d3pm}
  \end{equation}
  By fixing the sign to $+$ or $-$ in the $\pm$ subscripts
  in~\eqref{e:star} and~\eqref{e:d3pm}, let
  $\delta_3'=\delta_-'\delta_/''''/\delta_{\vartheta}'$ and
  $\delta_3''=\delta_+'\delta_/''''/\delta_{\vartheta}'$.  Now bound
  $\delta_3'$ and $\delta_3''$ below by $\delta_3^-$ and above by
  $\delta_3^+$, by combining the appropriate lower and upper bounds
  for $\delta_{\pm}'$, $\delta_{\vartheta}'$ from
  Lemma~\ref{l:sectheta}, and $\delta_/''''$, where
  \begin{displaymath}
    \frac{(1-\varepsilon_{\text{\texttt{T}}}^{})(1-\varepsilon)\cdot(1-\varepsilon)}{\sqrt{((1+\varepsilon)^2+1)/2}(1+\varepsilon)}=\delta_3^-\le\delta_3',\delta_3''\le\delta_3^+=\frac{(1+\varepsilon_{\text{\texttt{T}}}^{})(1+\varepsilon)\cdot(1+\varepsilon)}{\sqrt{((1-\varepsilon)^2+1)/2}(1-\varepsilon)},
  \end{displaymath}
  what proves~\eqref{e:d3}, by minimizing the numerator and maximizing
  the denominator to minimize the expression for $\delta_3'$ and
  $\delta_3''$, and vice versa, as in the proof of Lemma~\ref{l:R}.
\end{proof}

\looseness=-1
Therefore, \emph{it is possible to compute $\underline{R}$ with high
relative accuracy in the absence of underflows}, in the working
precision only, but it is easier to employ two multiplications, one
addition or subtraction, and one division in a wider, more precise
datatype.  Table~\ref{t:d} shows by how many $\varepsilon$s the lower
and the upper bounds for $\delta_1$, $\delta_2$, and $\delta_3$
from~\eqref{e:d1}, \eqref{e:d2}, and~\eqref{e:d3}, respectively,
differ from unity.  The quantities in the table's header were computed
symbolically as algebraic expressions by substituting $\varepsilon$
and $\varepsilon_{\text{\texttt{T}}}$ with their defining powers of
two, then approximated numerically with $p_{\text{\texttt{T}}}$
decimal digits, and rounded upwards to nine digits after the decimal
point, by a Wolfram Language script\footnote{The \texttt{relerr.wls}
file in the code supplement, executed by the Wolfram Engine 14.0.0 for
macOS (Intel).}.

\begin{table}[hbtp]
\caption{Lower and upper bounds for $\delta_1$, $\delta_2$, and $\delta_3$ in single and double precision.}\label{t:d}
{\addtolength{\tabcolsep}{-.5pt}
\begin{tabular}{@{}ccccccc@{}}
\toprule precision &
$(1-\delta_1^-)/\varepsilon$ & $(\delta_1^+-1)/\varepsilon$ &
$(1-\delta_2^-)/\varepsilon$ & $(\delta_2^+-1)/\varepsilon$ &
$(1-\delta_3^-)/\varepsilon$ & $(\delta_3^+-1)/\varepsilon$\\
\midrule
single & 3.499999665 & 3.500000336 & 4.499999457 & 4.500000544 & 3.499999667 & 3.500000338\\
double & 3.500000000 & 3.500000001 & 4.500000000 & 4.500000001 & 3.500000000 & 3.500000001\\
\botrule
\end{tabular}}
\end{table}
\subsection{The SVD of $R$ and $G$}\label{ss:3.3}
For $\underline{R}$ now holds $\mathfrak{t}''=13$.  If
$\underline{r_{11}}<\underline{r_{22}}$, the diagonal elements of
$\underline{R}$ have to be swapped, similarly to \texttt{xLASV2}.
This is done by symmetrically permuting $\underline{R}^T$ with
$\widetilde{P}=\left[\begin{smallmatrix}1&0\\0&1\end{smallmatrix}\right]$.
Multiplying
$\widetilde{R}=\widetilde{P}^T R^T \widetilde{P}=\widetilde{U}\widetilde{\Sigma}\widetilde{V}^T$
by $\widetilde{P}$ from the left and $\widetilde{P}^T$ from the right gives
\begin{displaymath}
  R^T=\widetilde{P}\widetilde{U}\widetilde{\Sigma}\widetilde{V}^T\widetilde{P}^T,\qquad
  R=\widetilde{P}\widetilde{V}\widetilde{\Sigma}\widetilde{U}^T\widetilde{P}^T=\check{U}\check{\Sigma}\check{V}^T,
\end{displaymath}
where $\check{U}=\widetilde{P}\widetilde{V}$,
$\check{\Sigma}=\widetilde{\Sigma}$, and
$\check{V}=\widetilde{P}\widetilde{U}$.  Therefore, having applied the
permutation $\widetilde{P}$,
\begin{equation}
  \begin{gathered}
    \widetilde{U}=\widetilde{U}_{\varphi}P_{\tilde{\bm{\sigma}}},\quad\widetilde{V}=\widetilde{V}_{\psi}P_{\tilde{\bm{\sigma}}},\qquad
    \check{U}=\check{U}_{\bm{\varphi}}P_{\tilde{\bm{\sigma}}},\quad\check{V}=\check{V}_{\bm{\psi}}P_{\tilde{\bm{\sigma}}},\\
    \begin{aligned}
      \check{U}_{\bm{\varphi}}&=\widetilde{P}\widetilde{V}_{\psi}=
      \begin{bmatrix}
      \sin\psi & \hphantom{-}\cos\psi\\
      \cos\psi & -\sin\psi
      \end{bmatrix}=
      \begin{bmatrix}
      \sin\psi & -\cos\psi\\
      \cos\psi & \hphantom{-}\sin\psi
      \end{bmatrix}S_2=
      \begin{bmatrix}
      \cos{\bm{\varphi}} & -\sin{\bm{\varphi}}\\
      \sin{\bm{\varphi}} & \hphantom{-}\cos{\bm{\varphi}}
      \end{bmatrix}S_2,\\
      \check{V}_{\bm{\psi}}&=\widetilde{P}\widetilde{U}_{\varphi}=
      \begin{bmatrix}
      \sin\varphi & \hphantom{-}\cos\varphi\\
      \cos\varphi & -\sin\varphi
      \end{bmatrix}=
      \begin{bmatrix}
      \sin\varphi & -\cos\varphi\\
      \cos\varphi & \hphantom{-}\sin\varphi
      \end{bmatrix}S_2=
      \begin{bmatrix}
      \cos{\bm{\psi}} & -\sin{\bm{\psi}}\\
      \sin{\bm{\psi}} & \hphantom{-}\cos{\bm{\psi}}
      \end{bmatrix}S_2,
    \end{aligned}\\
    S_2=\begin{bmatrix}
    1 & \hphantom{-}0\\
    0 & -1
    \end{bmatrix},\qquad
    \begin{gathered}
      \cos{\bm{\varphi}}=\sin\psi,\quad\sin{\bm{\varphi}}=\cos\psi,\quad\tan{\bm{\varphi}}=1/\tan\psi,\\
      \cos{\bm{\psi}}=\sin\varphi,\quad\sin{\bm{\psi}}=\cos\varphi,\quad\tan{\bm{\psi}}=1/\tan\varphi,
    \end{gathered}
  \end{gathered}
  \label{e:RT}
\end{equation}
where $\widetilde{U}_{\varphi}$, $\widetilde{V}_{\psi}$, and
$P_{\tilde{\bm{\sigma}}}$ come from the SVD of $\widetilde{R}$ in
Section~\ref{sss:3.3.1}.  If $\underline{r_{11}}\ge\underline{r_{22}}$
then $\underline{\widetilde{R}}=\underline{R}$,
$\check{U}=\widetilde{U}$, $\check{V}=\widetilde{V}$, \eqref{e:RT} is
not used, and let $S_2=I$, $\bm{\varphi}=\varphi$,
and $\bm{\psi}=\psi$.

Assume that no inexact underflow occurs in the computation of
$\underline{R}$.  If the initial matrix $G$ was triangular, let
$\delta_{11}=\delta_{12}=\delta_{22}=1$.  Else, let $\delta_{11}$,
$\delta_{12}$, and $\delta_{22}$ stand for the error factors of the
elements of $\underline{\widetilde{R}}$ that correspond to those from
Theorem~\ref{t:R}, i.e.,
\begin{equation}
  \widetilde{R}=\begin{bmatrix}
  \tilde{r}_{11} & \tilde{r}_{12}\\
  0 & \tilde{r}_{22}
  \end{bmatrix},\qquad
  \underline{\widetilde{R}}=\begin{bmatrix}
  \underline{\tilde{r}_{11}} & \underline{\tilde{r}_{12}}\\
  0 & \underline{\tilde{r}_{22}}
  \end{bmatrix}=\begin{bmatrix}
  \tilde{r}_{11}\delta_{11} & \tilde{r}_{12}\delta_{12}^{}\\
  0 & \tilde{r}_{22}\delta_{22}^{}
  \end{bmatrix},\quad
  \underline{\tilde{r}_{11}}\ge\underline{\tilde{r}_{22}}.
  \label{e:dR}
\end{equation}
It remains to compute the SVD of $\underline{\widetilde{R}}$ by an
alternative to \texttt{xLASV2}, what is described in
Section~\ref{sss:3.3.1}, and assemble the SVD of $G$, what is
explained in Section~\ref{sss:3.3.2}.
\subsubsection{The SVD of $\widetilde{R}$}\label{sss:3.3.1}
The key observation in this part is that the
traditional~\cite[Eq.~(4.12)]{Charlier-et-al-87} formula for
$\tan(2\varphi)/2$ involving the squares of the elements of
$\widetilde{R}$ can be simplified to an expression that does not
require any explicit squaring if the hypotenuse calculation is
considered a basic arithmetic operation.  The two following
expressions for $\tan(2\varphi)$ are equivalent,
\begin{equation}
  \tan(2\varphi)=\frac{2\tilde{r}_{12}^{}\tilde{r}_{22}^{}}{\tilde{r}_{11}^2+\tilde{r}_{12}^2-\tilde{r}_{22}^2}=\frac{2\tilde{r}_{12}^{}\tilde{r}_{22}^{}}{(h-\tilde{r}_{22}^{})(h+\tilde{r}_{22}^{})},\quad
  h=\sqrt{\tilde{r}_{11}^2+\tilde{r}_{12}^2}.
  \label{e:t2f}
\end{equation}

With
$\underline{h}=\hypot(\underline{\tilde{r}_{11}},\underline{\tilde{r}_{12}})$, let
$\mathbf{s}=\underline{h}\oplus\underline{\tilde{r}_{22}}$ and
$\mathbf{d}=\underline{h}\ominus\underline{\tilde{r}_{22}}$ be the sum
and the difference\footnote{With the prescaling as employed, $\ominus$ can be replaced
by subtraction $d=h-\tilde{r}_{22}$ and $\mathbf{d}=(e_d,f_d)$.} of
$\underline{h}$ and $\underline{\tilde{r}_{22}}$ as in~\eqref{e:plus}
and \eqref{e:minus}, respectively, for
$\underline{h}>\underline{\tilde{r}_{22}}$.  From~\eqref{e:dR} and
since $0<\tilde{r}_{ij}\le\nu/(2\sqrt{2})$ for $1\le i\le j\le 2$, it
holds
$0<\underline{\tilde{r}_{22}}\le\underline{\tilde{r}_{11}}\le\underline{h}\le\nu$.
Thus~\eqref{e:t2f} can be re-written using~\eqref{e:unary}
and~\eqref{e:binary}, with
$\tilde{\mathbf{r}}_{12}=(e_{\underline{\tilde{r}_{12}}},f_{\underline{\tilde{r}_{12}}})$
and
$\tilde{\mathbf{r}}_{22}=(e_{\underline{\tilde{r}_{22}}},f_{\underline{\tilde{r}_{22}}})$,
as
\begin{equation}
  \underline{h}=\underline{\tilde{r}_{22}}\implies
  \underline{\tan(2\varphi)}=\infty,\quad
  \underline{h}>\underline{\tilde{r}_{22}}\implies
  \underline{\tan(2\varphi)}=\fl((2\odot\tilde{\mathbf{r}}_{12}\odot\tilde{\mathbf{r}}_{22})\oslash(\mathbf{d}\odot\mathbf{s})),
  \label{e:t2f'}
\end{equation}
where the computation's precision is unchanged, but the exponent range
is widened.

In~\eqref{e:trad}, the denominator of the expression for
$\tan(2\phi)$,
$\mathsf{d}=g_{11}^2+g_{12}^2-g_{21}^2-g_{22}^2$,
can also be computed using $\hypot$, without explicitly squaring
any matrix element, as
\begin{displaymath}
    \mathsf{d}=\left(\sqrt{g_{11}^2+g_{12}^2}-\sqrt{g_{21}^2+g_{22}^2}\right)\left(\sqrt{g_{11}^2+g_{12}^2}+\sqrt{g_{21}^2+g_{22}^2}\right).
\end{displaymath}

Only if $\mathfrak{t}'\cong 13$ can happen that
$\fl(\underline{\tilde{r}_{11}}/\underline{\tilde{r}_{12}})<\varepsilon$.
In the first denominator in~\eqref{e:t2f},
$\tilde{r}_{11}^2$ and $\tilde{r}_{22}^2$ then have a negligible
effect on $\tilde{r}_{12}^2$, so the expression for
$\underline{\tan(2\varphi)}$ can be simplified, as in \texttt{xLASV2},
to the same formula which would the case
$\underline{\tilde{r}_{11}}=\underline{\tilde{r}_{22}}$ imply,
\begin{equation}
  \underline{\tan(2\varphi)}=\fl((2\underline{\tilde{r}_{22}})/\underline{\tilde{r}_{12}}).
  \label{e:t2f''}
\end{equation}

Let
$\tilde{\mathbf{r}}_{11}=(e_{\underline{\tilde{r}_{11}}},f_{\underline{\tilde{r}_{11}}})$.
If $\underline{\tilde{r}_{12}}=\underline{\tilde{r}_{22}}$,
\eqref{e:t2f} can be simplified by explicit squaring to
\begin{equation}
  \underline{\tan(2\varphi)}=\fl((2\odot\tilde{\mathbf{r}}_{12}\odot\tilde{\mathbf{r}}_{22})\oslash(\tilde{\mathbf{r}}_{11}\odot\tilde{\mathbf{r}}_{11})).
  \label{e:t2f'''}
\end{equation}
Both~\eqref{e:t2f''} and~\eqref{e:t2f'''} admit a simple roundoff
analysis.  However, \eqref{e:t2f'} does not, due to a subtraction of
potentially inexact values of a similar magnitude when computing
$\mathbf{d}$.  Section~\ref{s:4} shows, with a high probability by an
exhaustive testing, that~\eqref{e:t2f'} does not cause excessive
relative errors in the singular values for $\mathfrak{t}'\cong 13$,
and neither for $\mathfrak{t}'=15$ if the range of the exponents of
the elements of input matrices is limited in width to about
$(e_{\nu}-e_{\mu})/2$.  If $\hypot$ is not correctly rounded, the
procedure from~\cite{Novakovic-20,Novakovic-Singer-22} for computing
$\underline{\tan(2\varphi)}$ without squaring the input values can be
adopted instead of~\eqref{e:t2f'}, as shown in Algorithm~\ref{a:tf},
but still without theoretical relative error bounds.

\begin{algorithm}[hbtp]
  \caption{Computation of the functions of $\varphi$ from $\underline{\widetilde{R}}$.}\label{a:tf}
  \begin{algorithmic}[1]
    \IF{$\underline{\tilde{r}_{11}}=\underline{\tilde{r}_{22}}$}
    \STATE{$\underline{\tan(2\varphi)}=\fl((2\underline{\tilde{r}_{22}})/\underline{\tilde{r}_{12}})$}
    \COMMENT{\eqref{e:t2f''}}
    \ELSIF{$\underline{\tilde{r}_{12}}=\underline{\tilde{r}_{22}}$}
    \STATE{$\underline{\tan(2\varphi)}=\fl((2\odot\tilde{\mathbf{r}}_{12}\odot\tilde{\mathbf{r}}_{22})\oslash(\tilde{\mathbf{r}}_{11}\odot\tilde{\mathbf{r}}_{11}))$}
    \COMMENT{\eqref{e:t2f'''}}
    \ELSIF[only if $\mathfrak{t}'\cong 13$]{$\fl(\underline{\tilde{r}_{11}}/\underline{\tilde{r}_{12}})<\varepsilon$}
    \STATE{$\underline{\tan(2\varphi)}=\fl((2\underline{\tilde{r}_{22}})/\underline{\tilde{r}_{12}})$}
    \COMMENT{\eqref{e:t2f''}}
    \ELSIF[\cite{Novakovic-20,Novakovic-Singer-22}]{$\hypot$ is not correctly rounded}
    \IF{$\underline{\tilde{r}_{11}}>\underline{\tilde{r}_{12}}$}
    \STATE{$\underline{x}=\fl(\underline{\tilde{r}_{12}}/\underline{\tilde{r}_{11}});\quad\underline{y}=\fl(\underline{\tilde{r}_{22}}/\underline{\tilde{r}_{11}})$}
    \STATE{$\underline{\tan(2\varphi)}=\fl(\fl((2\underline{x})\underline{y})/\max(\fma(\fl(\underline{x}-\underline{y}),\fl(\underline{x}+\underline{y}),1),0))$}
    \ELSE[$\underline{\tilde{r}_{11}}\le\underline{\tilde{r}_{12}}$]
    \STATE{$\underline{x}=\fl(\underline{\tilde{r}_{11}}/\underline{\tilde{r}_{12}});\quad\underline{y}=\fl(\underline{\tilde{r}_{22}}/\underline{\tilde{r}_{12}})$}
    \STATE{$\underline{\tan(2\varphi)}=\fl((2\underline{y})/\max(\fma(\fl(\underline{x}-\underline{y}),\fl(\underline{x}+\underline{y}),1),0))$}
    \ENDIF
    \ELSE[the general case]
    \STATE{$\underline{h}=\hypot(\underline{\tilde{r}_{11}},\underline{\tilde{r}_{12}})$}
    \IF{$\underline{h}=\underline{\tilde{r}_{22}}$}
    \STATE{$\underline{\tan(2\varphi)}=\infty$}
    \COMMENT{\eqref{e:t2f'}}
    \ELSE[$\underline{h}>\underline{\tilde{r}_{22}}$]
    \STATE{$\mathbf{s}=\underline{h}\oplus\underline{\tilde{r}_{22}};\quad\mathbf{d}=\underline{h}\ominus\underline{\tilde{r}_{22}};\quad\underline{\tan(2\varphi)}=\fl((2\odot\tilde{\mathbf{r}}_{12}\odot\tilde{\mathbf{r}}_{22})\oslash(\mathbf{d}\odot\mathbf{s}))$}
    \COMMENT{\eqref{e:t2f'}}
    \ENDIF
    \ENDIF
    \IF{$\underline{\tan(2\varphi)}=\infty$}
    \STATE{$\underline{\tan\varphi}=1$}
    \ELSE[the general case]
    \STATE{$\underline{\tan\varphi}=\fl(\underline{\tan(2\varphi)}/\fl(1+\hypot(\underline{\tan(2\varphi)},1)))$}
    \COMMENT{\eqref{e:tf}}
    \ENDIF
    \STATE{$\underline{\sec\varphi}=\hypot(\underline{\tan\varphi},1);\quad\underline{\cos\varphi}=\fl(1/\underline{\sec\varphi});\quad\underline{\sin\varphi}=\fl(\underline{\tan\varphi}/\underline{\sec\varphi})$}
  \end{algorithmic}
\end{algorithm}

All cases of Algorithm~\ref{a:tf} lead to
$0\le\underline{\tan\varphi}\le 1$.  From $\underline{\tan\varphi}$
follows $\underline{\sec\varphi}$, as well as
$\underline{\cos\varphi}$ and $\underline{\sin\varphi}$, when
explicitly required, what completely determines
$\underline{\widetilde{U}_{\varphi}}$.

To determine $\underline{\widetilde{V}_{\psi}}$, $\tan\psi$ is
obtained from $\tan\varphi$ (see~\cite[Eq.~(4.10)]{Charlier-et-al-87})
as
\begin{equation}
  \tan\psi=(\tilde{r}_{12}+\tilde{r}_{22}\tan\varphi)/\tilde{r}_{11},\qquad
  \underline{\tan\psi}=\fl(\underline{t}/\underline{\tilde{r}_{11}}),\quad
  \underline{t}=\fma(\underline{\tilde{r}_{22}},\underline{\tan\varphi},\underline{\tilde{r}_{12}}).
  \label{e:tp}
\end{equation}
Let
$\mathop{\mathbf{sec}}\varphi=(e_{\underline{\sec\varphi}},f_{\underline{\sec\varphi}})$.
If $\underline{\tan\psi}$ is finite (e.g., when $\mathfrak{t}'=15$,
due to the pivoting~\cite[Theorem~1]{Novakovic-20}), so is
$\underline{\sec\psi}$.  Then, let
$\mathop{\mathbf{sec}}\psi=(e_{\underline{\sec\psi}},f_{\underline{\sec\psi}})$.
By fixing the evaluation order for reproducibility, the singular
values $\tilde{\bm{\sigma}}_1''$ and $\tilde{\bm{\sigma}}_2''$ of
$\underline{\widetilde{R}}$ are
computed~\cite{Hari-Matejas-09,Novakovic-Singer-22} as
\begin{equation}
  \mathbf{s}_{\psi}^{\varphi}=\mathop{\mathbf{sec}}\varphi\oslash\mathop{\mathbf{sec}}\psi,\quad
  \tilde{\bm{\sigma}}_2''=\tilde{\mathbf{r}}_{22}^{}\odot\mathbf{s}_{\psi}^{\varphi},\quad
  \tilde{\bm{\sigma}}_1''=\tilde{\mathbf{r}}_{11}^{}\oslash\mathbf{s}_{\psi}^{\varphi}.
  \label{e:s'}
\end{equation}

If $\underline{\tan\psi}$ overflows due to a small
$\underline{\tilde{r}_{11}}$ (the prescaling ensures that
$\underline{t}$ is always finite), let
$\mathbf{t}=(e_{\underline{t}},f_{\underline{t}})$.  In this case,
similarly to the one in \texttt{xLASV2} for
$\fl(\underline{\tilde{r}_{11}}/\underline{\tilde{r}_{12}})<\varepsilon$,
it holds $\sec\psi\gtrapprox\tan\psi$, so
$\cos\psi\lessapprox 1/\tan\psi$.  To confine subnormal values to
outputs only, let
\begin{equation}
  \mathop{\mathbf{cos}}\psi=\tilde{\mathbf{r}}_{11}\oslash\mathbf{t},\quad
  \underline{\cos\psi}=\fl(\mathop{\mathbf{cos}}\psi),\quad
  \underline{\sin\psi}=1.
  \label{e:tp'}
\end{equation}
By substituting
$1/\mathop{\mathbf{cos}}\psi\approx\mathbf{t}\oslash\tilde{\mathbf{r}}_{11}$
from~\eqref{e:tp'} for $\mathop{\mathbf{sec}}\psi$ in~\eqref{e:s'},
simplifying the results, and fixing the evaluation order, the singular
values of $\underline{\widetilde{R}}$ in this case are obtained as
\begin{equation}
  \tilde{\bm{\sigma}}_1''=\mathbf{t}\oslash\mathop{\mathbf{sec}}\varphi,\quad
  \tilde{\bm{\sigma}}_2''=\tilde{\mathbf{r}}_{22}^{}\odot(\mathop{\mathbf{sec}}\varphi\odot\mathop{\mathbf{cos}}\psi).
  \label{e:s''}
\end{equation}

From~\eqref{e:tp}, $\tan\psi>\nu$ implies
$\tan\varphi\lessapprox\tan(2\varphi)/2\lessapprox\tilde{r}_{22}/\tilde{r}_{12}\le\tilde{r}_{11}/\tilde{r}_{12}<1/(\nu-1)$,
so $\sec\varphi\gtrapprox 1$.  Therefore,
$\mathop{\mathbf{sec}}\varphi$ may be eliminated from~\eqref{e:s''},
similarly as in \texttt{xLASV2}.

The SVD of $\underline{\widetilde{R}}$ has thus been computed (without
explicitly forming $\underline{\widetilde{U}_{\varphi}}$ and
$\underline{\widetilde{V}_{\psi}}$) as
\begin{equation}
  \begin{aligned}
    \underline{\widetilde{R}}&\approx
    \begin{bmatrix}
      \underline{\cos\varphi} & -\underline{\sin\varphi}\\
      \underline{\sin\varphi} & \hphantom{-}\underline{\cos\varphi}
    \end{bmatrix}
    P_{\tilde{\bm{\sigma}}}^{}P_{\tilde{\bm{\sigma}}}^T
    \begin{bmatrix}
      \tilde{\bm{\sigma}}_1'' & 0\\
      0 & \tilde{\bm{\sigma}}_2''
    \end{bmatrix}
    P_{\tilde{\bm{\sigma}}}^{}P_{\tilde{\bm{\sigma}}}^T
    \begin{bmatrix}
      \hphantom{-}\underline{\cos\psi} & \underline{\sin\psi}\\
      -\underline{\sin\psi} & \underline{\cos\psi}
    \end{bmatrix}\\
    &=(\underline{\widetilde{U}_{\varphi}^{}}P_{\tilde{\bm{\sigma}}}^{})(P_{\tilde{\bm{\sigma}}}^T\underline{\widetilde{\Sigma}_{\tilde{\bm{\sigma}}}''}P_{\tilde{\bm{\sigma}}}^{})(P_{\tilde{\bm{\sigma}}}^T\underline{\widetilde{V}_{\psi}^T})
    =\underline{\widetilde{U}}\underline{\widetilde{\Sigma}_{\tilde{\bm{\sigma}}}'}\underline{\widetilde{V}^T}.
  \end{aligned}
  \label{e:RSVD}
\end{equation}
If $\tilde{\bm{\sigma}}_1''\prec\tilde{\bm{\sigma}}_2''$, then
$\tilde{\bm{\sigma}}_1'=\tilde{\bm{\sigma}}_2''$,
$\tilde{\bm{\sigma}}_2'=\tilde{\bm{\sigma}}_1''$, and
$P_{\tilde{\bm{\sigma}}}^{}=\left[\begin{smallmatrix}0&1\\1&0\end{smallmatrix}\right]$,
else $\tilde{\bm{\sigma}}_i'=\tilde{\bm{\sigma}}_i''$ and
$P_{\tilde{\bm{\sigma}}}^{}=I$, as presented in Algorithm~\ref{a:tp}.
If $\underline{r_{11}}<\underline{r_{22}}$ then
$\underline{\check{V}}$ should be formed as in~\eqref{e:RT}, and
$\underline{\check{U}}$ as well if $\mathfrak{t}'\cong 13$.  Else, if
$\underline{r_{11}}\ge\underline{r_{22}}$, then
$\underline{\check{V}}=\underline{\widetilde{V}}$, and (only
implicitly for $\mathfrak{t}'=15$)
$\underline{\check{U}}=\underline{\widetilde{U}}$.

\begin{algorithm}[hbtp]
  \caption{Computation of the functions of $\psi$ and the singular values of $\underline{\widetilde{R}}$.}\label{a:tp}
  \begin{algorithmic}[1]
    \STATE{$\underline{t}=\fma(\underline{\tilde{r}_{22}},\underline{\tan\varphi},\underline{\tilde{r}_{12}});\quad\underline{\tan\psi}=\fl(\underline{t}/\underline{\tilde{r}_{11}})$}
    \COMMENT{\eqref{e:tp}}
    \IF[only if $\mathfrak{t}'\cong 13$]{$\underline{\tan\psi}=\infty$}
    \STATE{$\mathbf{t}=(e_{\underline{t}},f_{\underline{t}});\quad\mathop{\mathbf{cos}}\psi=\tilde{\mathbf{r}}_{11}\oslash\mathbf{t};\quad\underline{\cos\psi}=\fl(\mathop{\mathbf{cos}}\psi);\quad\underline{\sin\psi}=1$}
    \COMMENT{\eqref{e:tp'}}
    \STATE{$\tilde{\bm{\sigma}}_1''=\mathbf{t};\quad\tilde{\bm{\sigma}}_2''=\tilde{\mathbf{r}}_{22}^{}\odot\mathop{\mathbf{cos}}\psi$}
    \COMMENT{\eqref{e:s''}}
    \ELSE[the general case]
    \STATE{$\underline{\sec\psi}=\hypot(\underline{\tan\psi},1);\quad\underline{\cos\psi}=\fl(1/\underline{\sec\psi});\quad\underline{\sin\psi}=\fl(\underline{\tan\psi}/\underline{\sec\psi})$}
    \STATE{$\mathop{\mathbf{sec}}\varphi=(e_{\underline{\sec\varphi}},f_{\underline{\sec\varphi}});\quad\mathop{\mathbf{sec}}\psi=(e_{\underline{\sec\psi}},f_{\underline{\sec\psi}});\quad\mathbf{s}_{\psi}^{\varphi}=\mathop{\mathbf{sec}}\varphi\oslash\mathop{\mathbf{sec}}\psi$}
    \STATE{$\tilde{\bm{\sigma}}_1''=\tilde{\mathbf{r}}_{11}^{}\oslash\mathbf{s}_{\psi}^{\varphi};\quad\tilde{\bm{\sigma}}_2''=\tilde{\mathbf{r}}_{22}^{}\odot\mathbf{s}_{\psi}^{\varphi}$}
    \COMMENT{\eqref{e:s'}}
    \ENDIF
    \IF[\eqref{e:prec}]{$\tilde{\bm{\sigma}}_1''\prec\tilde{\bm{\sigma}}_2''$}
    \STATE{$\tilde{\bm{\sigma}}_1'=\tilde{\bm{\sigma}}_2'';\quad\tilde{\bm{\sigma}}_2'=\tilde{\bm{\sigma}}_1'';\quad P_{\tilde{\bm{\sigma}}}^{}=\left[\begin{smallmatrix}0&1\\1&0\end{smallmatrix}\right]$}
    \ELSE[the general case]
    \STATE{$\tilde{\bm{\sigma}}_1'=\tilde{\bm{\sigma}}_1'';\quad\tilde{\bm{\sigma}}_2'=\tilde{\bm{\sigma}}_2'';\quad P_{\tilde{\bm{\sigma}}}^{}=\left[\begin{smallmatrix}1&0\\0&1\end{smallmatrix}\right]$}
    \ENDIF
  \end{algorithmic}
\end{algorithm}
\subsubsection{The SVD of $G$}\label{sss:3.3.2}
The approximate backscaled singular values of $G$ are
$\bm{\sigma}_i^{}=2^{-s}\odot\tilde{\bm{\sigma}}_i'$.  They should
remain in the exponent-``mantissa'' form if possible, to avoid
overflows and underflows.

\looseness=-1
Recall that, for $\mathfrak{t}'=15$,
$\underline{\widetilde{U}_{\varphi}^T}$ and the QR rotation
$\underline{U_{\vartheta}^T}$ have not been explicitly formed.  The
reason is that
$\underline{\widehat{U}^T}=\underline{\check{U}_{\bm{\varphi}}^T}\underline{U_+^T}$,
where $\underline{U_+^T}$ is constructed from
$\underline{U_{\vartheta}^T}$ as in~\eqref{e:UV15}, requires a
matrix-matrix multiplication that can and sporadically will degrade
the numerical orthogonality of $\underline{\widehat{U}}$.  On its own,
such a problem is expected and can be tolerated, but if the left
singular vectors of a pivot submatrix are applied to a pair of pivot
rows of a large iteration matrix, many times throughout the
Kogbetliantz process~\eqref{e:SVD}, it is imperative to make the vectors as
orthogonal as possible, and thus try not to destroy the singular
values of the iteration matrix.  In the following,
$\underline{\widehat{U}}$ is generated from a \emph{single}
$\tan{\bm{\phi}}$, where $\tan{\bm{\phi}}$ is a function of the
already computed $\tan{\bm{\varphi}}$ and $\tan{\vartheta}$.

If $\mathfrak{t}'\cong 13$, let
$\underline{\widehat{U}}=U_+^{}\underline{\check{U}}$, where
$U_+^{}$ comes from Section~\ref{sss:3.2.1}.  Else, due
to~\eqref{e:RT}, if $S_{22}^T=I$ in~\eqref{e:UV15}, the product
$\check{U}_{\bm{\varphi}}^TU_{\vartheta}^T=U_{\bm{\varphi}+\vartheta}^T$
can be written in the terms of $\bm{\varphi}+\vartheta$ as
\begin{equation}
  U_{\bm{\varphi}+\vartheta}^T=S_2^T
  \begin{bmatrix}
    \hphantom{-}\cos{\bm{\varphi}} & \sin{\bm{\varphi}}\\
    -\sin{\bm{\varphi}} & \cos{\bm{\varphi}}
  \end{bmatrix}
  \begin{bmatrix}
    \hphantom{-}\cos\vartheta & \sin\vartheta\\
    -\sin\vartheta & \cos\vartheta
  \end{bmatrix}=S_2^T
  \begin{bmatrix}
    \hphantom{-}\cos(\bm{\varphi}+\vartheta) & \sin(\bm{\varphi}+\vartheta)\\
    -\sin(\bm{\varphi}+\vartheta) & \cos(\bm{\varphi}+\vartheta)
  \end{bmatrix}.
  \label{e:p+t}
\end{equation}
If
$S_{22}^T=\left[\begin{smallmatrix}1&\hphantom{-}0\\0&-1\end{smallmatrix}\right]$,
$U_{\bm{\varphi}-\vartheta}^T=\check{U}_{\bm{\varphi}}^TS_{22}^TU_{\vartheta}^T$
can be written in the terms of $\bm{\varphi}-\vartheta$ as
\begin{equation}
  U_{\bm{\varphi}-\vartheta}^T=S_2^T
  \begin{bmatrix}
    \hphantom{-}\cos(\bm{\varphi}-\vartheta) & -\sin(\bm{\varphi}-\vartheta)\\
    -\sin(\bm{\varphi}-\vartheta) & -\cos(\bm{\varphi}-\vartheta)
  \end{bmatrix}=S_2^T
  \begin{bmatrix}
    \hphantom{-}\cos(\bm{\varphi}-\vartheta) & \sin(\bm{\varphi}-\vartheta)\\
    -\sin(\bm{\varphi}-\vartheta) & \cos(\bm{\varphi}-\vartheta)
  \end{bmatrix}S_2^{},
  \label{e:p-t}
\end{equation}
what is obtained by multiplying the matrices
$U_{\varphi}^T\left[\begin{smallmatrix}1&\hphantom{-}0\\0&-1\end{smallmatrix}\right]U_{\vartheta}^T$
and simplifying the result using the trigonometric identities for the
(co)sine of the difference of the angles $\bm{\varphi}$ and
$\vartheta$.  The middle matrix factor represents a possible sign
change of $r_{22}'$ as in Section~\ref{sss:3.2.2}.  The matrices
defined in~\eqref{e:p+t} and~\eqref{e:p-t} are determined by
$\tan(\bm{\varphi}+\vartheta)$ and $\tan(\bm{\varphi}-\vartheta)$,
respectively, where these tangents follow from the already computed
ones as
\begin{equation}
  \tan(\bm{\varphi}+\vartheta)=\frac{\tan\bm{\varphi}+\tan\vartheta}{1-\tan{\bm{\varphi}}\tan\vartheta},\qquad
  \tan(\bm{\varphi}-\vartheta)=\frac{\tan\bm{\varphi}-\tan\vartheta}{1+\tan{\bm{\varphi}}\tan\vartheta}.
  \label{e:tan+-}
\end{equation}
Finally, from~\eqref{e:UV15} and~\eqref{e:RT}, using
either~\eqref{e:p+t} or~\eqref{e:p-t}, the SVD of $G$ is completed as
\begin{equation}
  \underline{\widehat{U}^T}=\underline{U_{\bm{\varphi}\pm\vartheta}^T}P_U^TS_1^T,\quad
  \underline{U}=\underline{\widehat{U}}P_{\tilde{\bm{\sigma}}},\qquad
  \underline{V}=\underline{\check{V}}P_{\tilde{\bm{\sigma}}}.
  \label{e:UV}
\end{equation}

For~\eqref{e:UV}, $P_U^TS_1^T$ from~\eqref{e:UV15} is explicitly built
and stored.  It contains exactly one $\pm 1$ element in each row and
column, while the other is zero.  Its multiplication by
$\underline{U_{\bm{\varphi}\pm\vartheta}^T}$ is thus performed
error-free.  The tangents computed as in~\eqref{e:tan+-} might be
relatively inaccurate in theory, but the transformations they define
via the cosines and the sines from either~\eqref{e:p+t}
or~\eqref{e:p-t} are numerically orthogonal in practice, as shown in
Section~\ref{s:4}.

This heuristic might become irrelevant if the $ab+cd$ floating-point
operation with a single rounding~\cite{Lauter-17} becomes supported in
hardware.  Then, each element of
$\underline{\check{U}_{\bm{\varphi}}^T}\underline{U_{\vartheta}^T}$ (a
product of two $2\times 2$ matrices) can be formed with one such
operation.  It remains to be seen if the multiplication approach
improves accuracy of the computed left singular vectors without
spoiling their orthogonality, compared to the proposed heuristic.

From the method's construction, it follows that if the side (left or
right) on which the signs are extracted while preparing $R$ is fixed
(see Section~\ref{ss:3.1}) and whenever the assumptions on the
arithmetic hold, the SVD of $G$ as proposed here is \emph{bitwise
reproducible} for any $G$ with finite elements.  Also, the method does
not produce any infinite or undefined element in its outputs $U$, $V$,
and (conditionally, as described) $\Sigma$.
\subsection{A complex input matrix}\label{ss:3.4}
If $G$ is purely imaginary, $\pm\mathrm{i}G$ is real.  Else, if $G$
has at least one complex element, the proposed real method is altered,
as detailed in~\cite{Novakovic-20,Novakovic-Singer-22}, in the
following ways:
\begin{enumerate}
  \item To make the element
    $0\ne g_{ij}=|g_{ij}|\mathrm{e}^{\mathrm{i}\alpha_{ij}}$ real and
    positive, its row or column is multiplied by
    $\mathrm{e}^{-\mathrm{i}\alpha_{ij}}$ (which goes into a sign
    matrix), and the element is replaced by its absolute value.  To
    avoid overflow, let $\mathfrak{s}_{\mathbb{C}}=\mathfrak{s}+1$
    in~\eqref{e:ts}.  The exponents of each component (real and
    imaginary) of every element are considered in~\eqref{e:eG}.
  \item $\underline{U_+^T}$ is explicitly constructed
    in~\eqref{e:UV15}, and
    $\underline{\check{U}_{\bm{\varphi}}^T}\underline{U_+^T}$ is
    formed by a real-complex matrix multiplication.  The correctly
    rounded $ab+cd$ operation~\cite{Lauter-17} would be helpful here.
    Merging
    $\underline{\check{U}_{\bm{\varphi}}^T}\underline{U_{\vartheta}^T}$
    as in~\eqref{e:p+t} or~\eqref{e:p-t} remains a possibility if
    $S_{22}^{}$ happens to be real.
  \item Since~\eqref{e:EXT} is no longer directly applicable for
    ensuring stability, no computation is performed in a wider
    datatype.  Reproducibility of the whole method is conditional upon
    reproducibility of the complex multiplication and the absolute
    value ($\hypot$).
\end{enumerate}
Once $\underline{R}$ is obtained, the algorithms from
Section~\ref{ss:3.3} work unchanged.
\section{Numerical testing}\label{s:4}
Numerical testing was performed on machines with a 64-core Intel Xeon
Phi 7210 CPU, a 64-bit Linux, and the Intel oneAPI Base and HPC
toolkits, version 2024.1.

Let the LAPACK's \texttt{xLASV2} routine be denoted by $\mathtt{L}$.
The Kogbetliantz SVD in the same datatype is denoted by $\mathtt{K}$.
Unless such information is clear from the context, let the results'
designators carry a subscript $\mathtt{K}$ or $\mathtt{L}$ in the
following figures, depending on the routine that computed them, and
also a superscript $\circ$ or $\bullet$, signifying how the input
matrices were generated.  All inputs were random.  Those denoted by
$\circ$ had their elements generated as Fortran's pseudorandom numbers
not above unity in magnitude, and those symbolized by $\bullet$ had
their elements' magnitudes in the ``safe'' range $[\mu,\nu/4]$, as
defined by~\eqref{e:safe}, to avoid overflows with $\mathtt{L}$ and
underflows due to the prescaling in $\mathtt{K}$.  The latter random
numbers were provided by the CPU's \texttt{rdrand} instructions.  If
not kept, the $\bullet$ inputs are thus not reproducible, unlike the
$\circ$ ones if the seed is preserved.

All relative error measures were computed in quadruple precision from
data in the working (single or double) precision.  The unknown exact
singular values of the input matrices were approximated by the
Kogbetliantz SVD method adapted to quadruple precision (with a
$\hypot$ operation that might not have been correctly rounded).

With $G$ given and $\underline{U}$, $\underline{\Sigma}$,
$\underline{V}$ computed, let the relative SVD residual be defined as
\begin{equation}
  \mathop{\mathrm{re}}G=\|G-\underline{U}\underline{\Sigma}\underline{V^T}\|_F/\|G\|_F,
  \label{e:svdres}
\end{equation}
the maximal relative error in the computed singular values
$\underline{\sigma_i}$ (with $\sigma_i$ being exact) as
\begin{equation}
  \mathop{\mathrm{re}}\sigma_i=|\sigma_i-\underline{\sigma_i}|/\sigma_i,\quad
  1\le i\le 2,\qquad
  \sigma_i=0\wedge\underline{\sigma_i}=0\implies\mathop{\mathrm{re}}\sigma_i=0,
  \label{e:relerr}
\end{equation}
and the departure from orthogonality in the Frobenius norm for
matrices of the left and right singular vectors (what can be seen as
the relative error with respect to $I$) as
\begin{equation}
  \mathop{\mathrm{re}}U=\|\underline{U}^T\underline{U}-I\|_F,\qquad
  \mathop{\mathrm{re}}V=\|\underline{V}^T\underline{V}-I\|_F.
  \label{e:ortho}
\end{equation}

Every datapoint in the figures shows the maximum of a particular
relative error measure over a batch of input matrices, were each batch
(run) contained $2^{30}$ matrices.

Figure~\ref{f:1} covers the case of upper triangular input matrices,
which can be processed by both $\mathtt{K}$ and $\mathtt{L}$, and the
measures~\eqref{e:svdres} and~\eqref{e:ortho}.  Numerical
orthogonality of the singular vectors computed by $\mathtt{K}$ is
noticeably better than of those obtained by $\mathtt{L}$, in the worst
case.  Also, the relative SVD residuals are slightly better, in the
$\bullet$ and the $\circ$ runs.

\begin{figure}[hbtp]
  \centering
  \includegraphics{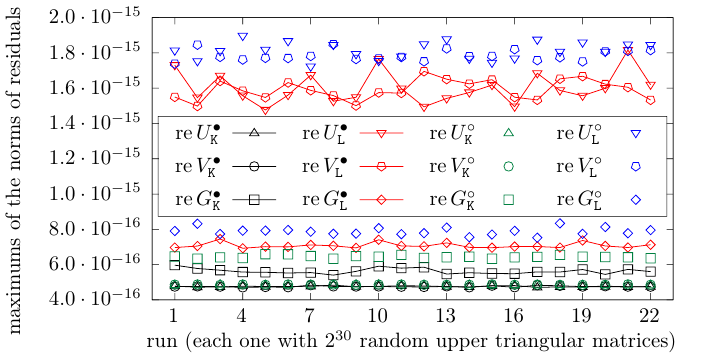}
  \caption{Numerical orthogonality of the singular vectors and the
    relative SVD residuals with $\mathtt{K}$ and $\mathtt{L}$ on
    random upper triangular double precision matrices.}\label{f:1}
\end{figure}

Figure~\ref{f:2} shows the relative errors in the singular
values~\eqref{e:relerr} of the same matrices from Figure~\ref{f:1}.
The unity mark for
$\mathop{\mathrm{re}_{\mathtt{L}}^{}}\sigma_2^{\bullet}$ indicates
that $\mathtt{L}$ can cause the relative errors in the smaller
singular values, $\underline{\sigma_2}$, to be so high in the
$\bullet$ case that their maximum was unity in all runs and cannot be
displayed in Figure~\ref{f:2}, most likely due to underflow to zero of
$\underline{\sigma_2^{\bullet}}$ when the ``exact''
$\sigma_2^{\bullet}>0$ in~\eqref{e:relerr}.  However, when
$\mathtt{L}$ managed to compute the smaller singular values accurately
in the $\circ$ case, the maximum of their relative errors was a bit
smaller than the one from $\mathtt{K}$, the cause of which is worth
exploring.  The same holds for the larger singular values, which were
computed accurately by both $\mathtt{L}$ and $\mathtt{K}$.

\begin{figure}[hbtp]
  \centering
  \includegraphics{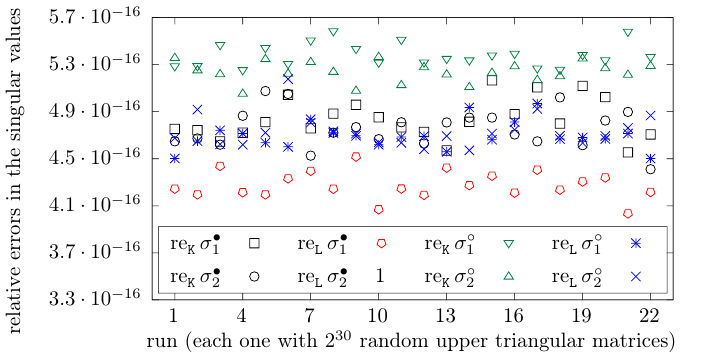}
  \caption{The relative errors in the singular values with
    $\mathtt{K}$ and $\mathtt{L}$ on random upper triangular double
    precision matrices, with
    $\max\kappa_2^{\bullet}\lessapprox 9.45\cdot 10^{1229}$ and
    $\max\kappa_2^{\circ}\lessapprox 7.32\cdot 10^{10}$.}\label{f:2}
\end{figure}

To put $\max\kappa_2^{\bullet}\lessapprox 9.45\cdot 10^{1229}$, for
which $\mathtt{K}$ still accurately computed all singular values (in
the exponent-``mantissa'' form, and thus not underflowing), into
perspective, the highest possible condition number for triangular
matrices in the $\bullet$ case can be estimated by recalling that
Algorithms~\ref{a:tf} and~\ref{a:tp} were also performed in quadruple
precision (to get $\sigma_1$, $\sigma_2$, and so $\kappa_2$), where
$\mu$ and $\nu$ of double precision, as well as $\nu/\mu$, are within
the normal range.  Then, $\tan\varphi$ can be made small and
$\tan\psi$ huge by, e.g.,
\begin{displaymath}
  G=\begin{bmatrix}
  \mu & \nu/4\\
  0 & \mu
  \end{bmatrix}\implies
  \tan(2\varphi)=\frac{8\mu}{\nu}\implies
  \frac{2\mu}{\nu}<\tan\varphi\lessapprox\frac{4\mu}{\nu}\implies
  \tan\psi\gtrapprox\frac{\nu}{4\mu}.
\end{displaymath}
Therefore, the condition number of $G$ is a cubic expression in
$\nu/\mu$, since, from~\eqref{e:s'},
\begin{displaymath}
  \sigma_2=\mu\frac{\sec\varphi}{\sec\psi}\approx\frac{4\mu^2}{\nu},\quad
  \sigma_1=\frac{\nu}{4}\frac{\sec\psi}{\sec\varphi}\approx\frac{\nu^2}{16\mu},\qquad
  \kappa_2=\frac{\sigma_1}{\sigma_2}\approx\frac{\nu^3}{64\mu^3}.
\end{displaymath}

Figure~\ref{f:3} focuses on $\mathtt{K}$ and general input matrices,
with all their elements random.  Inaccuracy of the smaller singular
values in the $\bullet$ case motivated the search for safe exponent
ranges of the elements of input matrices that should preserve accuracy
of $\underline{\sigma_2}$ from $\mathtt{L}$ for $\mathfrak{t}=13$ and
from $\mathtt{K}$ for $\mathfrak{t}=15$.  For that, the range of
random values was restricted, and only those outputs $x$ from
\texttt{rdrand} for which $|x|\in[2^{\varsigma}\mu,\nu/4]$ were
accepted, where $\varsigma$ was a positive integer parameter
independently chosen for each run.

\begin{figure}[hbtp]
  \centering
  \includegraphics{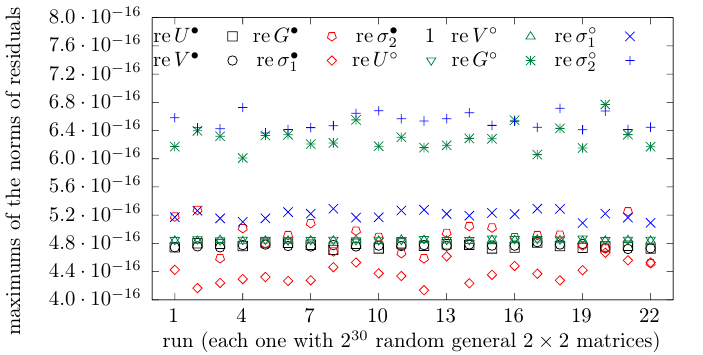}
  \caption{Numerical orthogonality of the singular vectors, the
    relative SVD residuals, and the relative errors in the singular
    values with $\mathtt{K}$ on random double precision matrices, with
    $\max\kappa_2^{\bullet}\lessapprox 1.42\cdot 10^{616}$ and
    $\max\kappa_2^{\circ}\lessapprox 2.84\cdot 10^{10}$.}\label{f:3}
\end{figure}

Figure~\ref{f:4} shows the results of this search for $\mathtt{K}$ and
$\mathtt{L}$.  Approximately half-way through the entire normal
exponent range the relative errors in the smaller singular values
stabilize to a single-digit multiple of $\varepsilon$.  Thus, when for
the exponents in $G$ holds
\begin{displaymath}
  \max_{1\le i,j\le 2}{e_{g_{ij}}}-\min_{1\le i,j\le 2}{e_{g_{ij}}}<(e_{\nu}-e_{\mu})/2
\end{displaymath}
(ignoring the exponent of $0$) it might be expected that $\mathtt{K}$
computes $\underline{\sigma_2}$ accurately, while $\mathtt{L}$ should
additionally be safeguarded by its user from the elements too close to
$\mu$.  The proposed prescaling, but with
$\mathfrak{s}_{\mathtt{L}}=\mathfrak{s}+1$ (or more), might be applied
to $G$ before $\mathtt{L}$.

\begin{figure}[hbtp]
  \centering
  \includegraphics{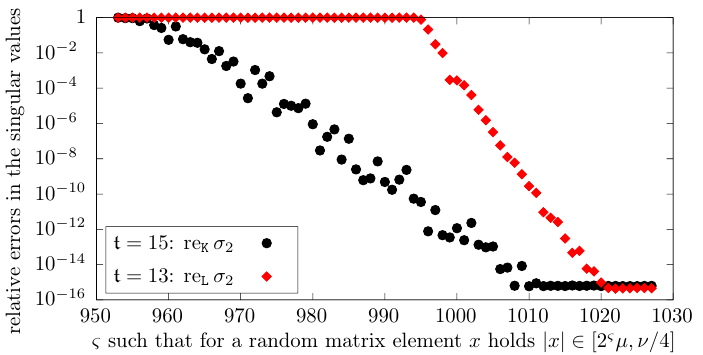}
  \caption{The observed decay of the relative errors in the smaller
    singular values by narrowing of the exponent range of the elements
    of double precision input matrices, where $\kappa_2^{\mathtt{K}}$
    falls from $1.59\cdot 10^{329}$ for $\varsigma=953$ to
    $2.33\cdot 10^{308}$ for $\varsigma=1027$, and
    $\kappa_2^{\mathtt{L}}$ from $2.23\cdot 10^{656}$ to
    $4.93\cdot 10^{611}$.}\label{f:4}
\end{figure}

A timing comparison of \texttt{xLASV2} ($\mathtt{L}$) and the proposed
method ($\mathtt{K}$) in single\footnote{These results were computed
using the Intel oneAPI Base and HPC toolkits, version 2025.0.4.} and
double precision is given in Table~\ref{t:time}.  Each method call was
timed by the \texttt{clock\_gettime} Linux function\footnote{See,
e.g., \url{https://man7.org/linux/man-pages/man2/clock_gettime.2.html}
URL\@.} using the \texttt{CLOCK\_MONOTONIC\_RAW} clock.  By
construction, $\mathtt{K}$ is more computationally complex than
$\mathtt{L}$, so it is not unexpected to be, on average, more than
twice slower.  In the Kogbetliantz SVD for matrices of order $n>2$,
each $2\times 2$ SVD, itself of a constant complexity, is followed by
(at most four) $\mathop{O}(n)$ matrix updates, and thus the run-time
of any reasonably fast $2\times 2$ SVD method is not a bottleneck.  A
sufficiently accurate SVD of order two might even reduce somewhat the
overall number of iterations of the Kogbetliantz SVD on certain
inputs, compared to using a less accurate but faster method, and, if
so, partly compensate for its relative slowness.

\begin{table}[hbtp]
\caption{Average run-times per matrix of the methods $\mathtt{K}$ and
  $\mathtt{L}$, on $\circ$ and $\bullet$ upper triangular random
  matrices ($22\cdot 2^{30}$ of each kind), and the slowdown of
  $\mathtt{K}$ versus $\mathtt{L}$.}\label{t:time}
{\addtolength{\tabcolsep}{-.5pt}
\begin{tabular}{@{}ccccc@{}}
\toprule method &
single precision $\circ$ & single precision $\bullet$ &
double precision $\circ$ & double precision $\bullet$\\
\midrule
$\mathtt{K}$ & $3.133\cdot 10^{-6}$~s & $2.998\cdot 10^{-6}$~s & $4.370\cdot 10^{-6}$~s & $3.640\cdot 10^{-6}$~s\\
$\mathtt{L}$ & $1.476\cdot 10^{-6}$~s & $1.668\cdot 10^{-6}$~s & $1.658\cdot 10^{-6}$~s & $1.653\cdot 10^{-6}$~s\\
\midrule
$\mathtt{K}/\mathtt{L}$ & $2.123$ & $1.797$ & $2.636$ & $2.202$\\
\botrule
\end{tabular}}
\end{table}

An unoptimized OpenMP-parallel implementation of the Kogbetliantz SVD
for $G$ of order $n>2$ with the scaling of $G$ in the spirit
of~\cite{Novakovic-23} but stronger (accounting for the two-sided
transformations of $G$) and the modified modulus pivot
strategy~\cite{Novakovic-Singer-11}, when run with 64 threads spread
across the CPU cores, a deterministic reduction procedure, and
\texttt{OMP\_DYNAMIC=FALSE}, showed up to 10\% speedup over the
one-sided Jacobi SVD routine without preconditioning,
\texttt{DGESVJ}~\cite{Drmac-97,Drmac-Veselic-08b}, from the threaded
Intel MKL library for large enough $n$ (up to $5376$), with the left
singular vectors from the former being a bit more orthogonal than the
ones from the latter, while the opposite was true for the right
singular vectors, on the highly conditioned input matrices
from~\cite{Novakovic-23}.  The singular values from \texttt{DGESVJ}
were less than an order of magnitude more accurate.
\section{Conclusions and future work}\label{s:5}
The proposed Kogbetliantz method for the SVD of order two computed
highly numerically orthogonal singular vectors in all tests.  The
larger singular values were relatively accurate up to a few
$\varepsilon$ in all tests, and the smaller ones were when the input
matrices were triangular, or, for the general (without zeros) input
matrices, if the range of their elements was narrower than or about
half of the width of the range of normal values.

The constituent phases of the method can be used on their own.  The
prescaling might help \texttt{xLASV2} when its inputs are small.  The
highly accurate triangularization might be combined with
\texttt{xLASV2} instead, as an alternative method for general
matrices.  And the proposed SVD of triangular matrices demonstrates
some of the benefits of the more complex correctly rounded operations
($\hypot$), but they go beyond that.

High relative accuracy for $\tan(2\varphi)$ from~\eqref{e:trad} might
be achieved, barring underflow, if the four-way fused dot product
operation $ab+cd+ef+gh$, DOT4, with a single rounding of the exact
value~\cite{Lutz-et-al-24}, becomes available in hardware.  Then the
denominator of the expression for $\tan(2\varphi)$ in~\eqref{e:trad}
could be computed, even without scaling if in a wider datatype, by the
DOT4, and the numerator by the DOT2 ($ab+cd$) operation.

The proposed heuristic for improving orthogonality of the left
singular vectors might be helpful in other cases when two plane
rotations have to be composed into one and the tangents of their
angles are known.  It already brings a slight advantage to the
Kogbetliantz SVD of order $n$ with respect to the one-sided Jacobi
SVD in this regard.

\looseness=-1
With a proper vectorization, and by removing all redundancies from the
preliminary implementation, it might be feasible to speed up the
Kogbetliantz SVD of order $n$ further, since adding more threads is
beneficial as long as their number is not above $\mathsf{n}$.
\backmatter
\begin{appendices}
\noindent\textbf{Supplementary information.}
The document \texttt{sm.pdf} supplements this paper with further
remarks on methods for larger matrices and the single precision
testing results.

\noindent\textbf{Acknowledgments.}
The author would like to thank Dean Singer for his material support,
to Vjeran Hari for fruitful discussions, and to the anonymous referee
for the suggestions that significantly improved the manuscript's
completeness and notation.
\section*{Declarations}
\noindent\textbf{Funding.}
This work was supported in part by Croatian Science Foundation under
the expired project IP--2014--09--3670 ``Matrix Factorizations and
Block Diagonalization Algorithms''
(\href{https://web.math.pmf.unizg.hr/mfbda}{MFBDA}), in the form of
unlimited compute time granted for the testing.

\noindent\textbf{Competing interests.}
The author has no relevant competing interests to declare.

\noindent\textbf{Code availability.}
The code is available in \url{https://github.com/venovako/KogAcc}
repository, and in the supporting \url{https://github.com/venovako/libpvn} repository.
\end{appendices}
%

\end{document}